\documentclass[11pt,leqno]{article}
\usepackage{mathrsfs,amssymb,amsmath,amsthm, color}
\usepackage[dvips]{graphicx}
\makeatletter
	
	\@addtoreset{equation}{section}
\makeatother
\allowdisplaybreaks[4]
\begin{document}

\renewcommand{\theenumi}{\rm (\roman{enumi})}
\renewcommand{\labelenumi}{\rm \theenumi}

\newtheorem{thm}{Theorem}[section]
\newtheorem{defi}[thm]{Definition}
\newtheorem{lem}[thm]{Lemma}
\newtheorem{prop}[thm]{Proposition}
\newtheorem{cor}[thm]{Corollary}
\newtheorem{exam}[thm]{Example}
\newtheorem{conj}[thm]{Conjecture}
\newtheorem{rem}[thm]{Remark}

\newcommand{\id}{\mathop{\text{\rm id}}}
\newcommand{\supp}{\mathop{\text{\rm supp}}}
\newcommand{\esssup}{\mathop{\rm ess\,sup}}

\title{Stochastic quantization associated with the $\exp(\Phi)_2$-quantum 
field model driven by space-time white noise on the torus}

\author{Masato Hoshino\hspace{0.7mm}$^{a)}$, Hiroshi Kawabi\hspace{0.7mm}$^{b)}$ and 
Seiichiro Kusuoka\hspace{0.7mm}$^{c)}$
\vspace{5mm}\\
\normalsize $^{a)}$ Faculty of Mathematics, Kyushu University,\\
\normalsize 744 Motooka, Nishi-ku, Fukuoka 819-0395, Japan\\
\normalsize e-mail address: {\tt{hoshino@math.kyushu-u.ac.jp}} \vspace{5mm}\\
\normalsize $^{b)}$ Department of Mathematics, Hiyoshi Campus, Keio University,\\
\normalsize  4-1-1 Hiyoshi, Kohoku-ku, Yokohama 223-8521, Japan \\
\normalsize e-mail address: {\tt{kawabi@keio.jp}} \vspace{5mm}\\
\normalsize $^{c)}$ Department of Mathematics, Graduate School of Science, Kyoto University,\\
\normalsize Kitashirakawa Oiwakecho, Sakyo-ku, Kyoto 606-8502, Japan\\
\normalsize e-mail address: {\tt{kusuoka@math.kyoto-u.ac.jp}}}
\maketitle

\begin{abstract}
We consider 
a quantum field model with 
exponential interactions on the two-dimensional torus, which is called 
the $\exp (\Phi)_{2}$-quantum field model or
H{\o}egh-Krohn's model.
In the present paper, we study the stochastic quantization of this model by 
singular stochastic partial differential equations, which is recently developed. 
By the method, we construct a unique time-global solution 
and the invariant probability measure of the corresponding stochastic quantization equation, 
and identify it with
an infinite-dimensional diffusion process, which has been constructed
by the Dirichlet form approach.
\end{abstract}
\section{Introduction}
\subsection{Background}
In recent years, there has been a growing interest in 
stochastic partial differential equations (SPDEs in short)
arising in various models of
Euclidean quantum field theory, hydrodynamics, statistical mechanics and so on.
One of the principal themes
in studies of SPDEs
is to construct a dynamical system whose invariant measure
is a given Gibbs measure on an infinite-dimensional state space 
through SPDEs.
In \cite{PW81}, Parisi and Wu proposed such a program for 
Euclidean quantum field theory, and this program is 
now called the
{\it{stochastic quantization}}. 
For a concise overview on the stochastic quantization,
we refer to \cite{AMR15, AK19, ADG19} and references therein.

In Euclidean quantum field theory, the $\Phi^{2m}_{2}$-{\it{quantum field}} ($m=2,3,\ldots$), 
a special case of the $P(\Phi)_{2}$-{\it{quantum fields}} in finite volume, is one of the most important objects
(see e.g., \cite{GJ86, Sim74}). It is a probability measure
on ${\mathcal D}'(\Lambda)$,
the space of distributions 
on the two-dimensional torus $\Lambda={\mathbb T}^{2}=
(\mathbb R/ 2\pi \mathbb Z)^{2}$, which is formally given by the expression
\begin{equation}
\mu^{(2m)}_{\sf pol}(d\phi)
\varpropto
\exp \Big \{ -
\frac{1}{2} 
\int_{\Lambda} 
\Big( m_{0}^{2}\phi(x)^{2}+\vert \nabla \phi(x) \vert
^{2} +2\phi(x)^{2m} \Big )dx
\Big \}
\prod_{x\in {\Lambda}} d\phi(x),
\label{formal-free}
\end{equation}
where $m_{0}>0$ and $\Delta$ are
mass and
the Laplacian in $L^{2}(\Lambda)$ with periodic boundary conditions, 
respectively. However, we need a
renormalization procedure for $\phi^{2m}$ to give a rigorous meaning to 
(\ref{formal-free}) because the power $\phi^{2m}$
is not defined for
$\phi\in {\mathcal D}'(\Lambda)$ in general.
We introduce the massive {\it{Gaussian free field}} $\mu_{0}$ by 
the Gaussian measure on ${\mathcal D}'(\Lambda)$
with zero mean and the covariance operator $(m_{0}^{2}-\Delta)^{-1}$, and 
replace $\phi^{2m}$ by the $2m$-th order
Wick power 
$(\phi^{2m})^{\diamond}$
with respect to $\mu_{0}$. 
Then the $\Phi^{2m}_{2}$-quantum field 
$\mu^{(2m)}_{\sf pol}$ is rigorously defined by
\begin{equation}
\mu^{(2m)}_{\sf pol}(d\phi)=\frac{1}{{Z}_{2m}}
\exp \Big (
- \int_{\Lambda}
(\phi^{2m})^{\diamond}(x) 
dx \Big )
\mu_{0}(d\phi), 
\label{Pphi2}
\end{equation}
where 
$Z_{2m}>0$ is the normalizing constant given by
\begin{equation}
Z_{2m}=\int_{ {\mathcal D}'(\Lambda)} \exp 
\Big(-\int_{\Lambda} 
(\phi^{2m})^{\diamond}(x) 
dx \Big) \mu_{0}(d\phi).
\nonumber
\end{equation}
Parisi and Wu \cite{PW81} first proposed 
an SPDE 
\begin{equation}
\partial_{t}\Phi_{t}(x)=
\frac{1}{2}(\Delta-m_{0}^{2})\Phi_{t}(x)
-m
(\Phi^{2m-1}_{t})^{\diamond}(x)
+
{\dot W}_{t}(x), 
\quad t>0,~x\in \Lambda,
\label{PW-SQE}
\end{equation}
which realizes the stochastic quantization for 
$\mu^{(2m)}_{\sf pol}$ by heuristic calculations. 
Here
$( \dot W_{t} )_{t\geq 0}$ is an $\mathbb R$-valued Gaussian
space-time white noise, that is, the time derivative of 
a standard $L^{2}(\Lambda)$-cylindrical Brownian motion
$\{ W_{t}=(W_{t}(x) )_{x\in {\Lambda}}\}_{t\geq 0}$.
This SPDE is called the ($P(\Phi)_{2}$-){\it{stochastic quantization equation}}.
Due to the singularity of the
nonlinear drift term, 
the interpretation and construction of a solution to this 
SPDE have been a challenging problem for many years.
In \cite{JM85}, 
Jona-Lasinio and Mitter studied
a modified SPDE
\begin{equation}\label{JM-SQE}\begin{array}{rl}
\partial_{t}\Phi_{t}(x)=&\displaystyle
-\frac{1}{2}(m_{0}^{2}-\Delta)^{\varepsilon} \Phi_{t}(x)
-m(m_{0}^{2}-\Delta)^{\varepsilon-1}
(\Phi^{2m-1}_{t})^{\diamond}(x)
\\
&\displaystyle
+(m_{0}^{2}-\Delta)^{\frac{\varepsilon-1}{2}}
{\dot W}_{t}(x), 
\quad t>0,
~x\in \Lambda,
\end{array}
\end{equation}
where $\varepsilon$ is a sufficiently small positive constant.
Note that $\mu^{(2m)}_{\sf pol}$ is also an invariant measure of (\ref{JM-SQE}).
This modification allows 
smoothing of both the nonlinear drift term and the driving noise term, and thus they could apply
the Girsanov transform for constructing a
solution to (\ref{JM-SQE}) in the weak sense.
Since then, there has been a large number of follow-up papers on 
the modified SPDE (\ref{JM-SQE}),
and 
both theories of SPDEs and Dirichlet forms on infinite-dimensional state spaces have been developed intensively 
(see e.g., \cite{BCM88, AR90, AR91, ARZ93, HK93, GG96}).

On the other hand, 
the Girsanov transform approach
does not work efficiently for solving the original SPDE (\ref{PW-SQE}) (i.e., the 
modified SPDE (\ref{JM-SQE}) in the case $\varepsilon=1$)
due to the singularity of the nonlinear drift term.
Applying the Dirichlet form theory,
Albeverio and R\"ockner \cite{AR91}
constructed a diffusion process 
solving (\ref{PW-SQE}) 
in the weak sense. 
Besides, Mikulevicius and Rozovskii \cite{MR99} 
developed their compactness method for SPDEs and
constructed martingale solutions of (\ref{JM-SQE}) for all
$0<\varepsilon \leq 1$. They also proved uniqueness in law for all $0<\varepsilon<1$.
Later in \cite{DPD03}, Da Prato and Debussche 
constructed a unique global solution to (\ref{PW-SQE}) 
in the strong probabilistic sense by splitting the original SPDE 
(\ref{PW-SQE}) into the Ornstein-Uhlenbeck process and the shifted equation. Since the
solution of the shifted equation is much smoother than the Ornstein-Uhlenbeck process,
they could solve the shifted equation by a fixed point argument on a suitable Besov space.
Their approach is now called the {\it{Da Prato--Debussche argument}}, and 
was applied to the infinite volume case in \cite{MW17a}.
In a recent paper \cite{RZZ17a}, R\"ockner, Zhu and Zhu 
obtained both restricted Markov uniqueness of the generator 
and the uniqueness of the martingale solution to (\ref{PW-SQE}) by
identifying
the solution 
obtained in \cite{DPD03} with one
obtained by the Dirichlet form approach. 

We should mention here that the 
{\it{$\Phi^{4}_{3}$-quantum field model}}  
in finite volume, heuristically given by
(\ref{Pphi2}) with $\Lambda={\mathbb T}^{3}=(\mathbb R/ 2\pi \mathbb Z)^{3}$ and 
$m=2$, has also been received
a lot of attention in the Euclidean quantum field theory. 
To make a rigorous meaning to the three-dimensional version of 
the probability measure $\mu^{(4)}_{\sf pol}(d\phi)$, 
we need a further renormalization procedure beyond the Wick renormalization (see e.g., \cite{BFS83} and references
therein).
For this reason, the stochastic quantization equation associated with the 
$\Phi^{4}_{3}$-quantum field model (i.e.,
the three-dimensional version of the SPDE (\ref{PW-SQE}) with $m=2$)
has not been
studied satisfactorily for a long time. After
Hairer's groundbreaking work on {\it{regularity structures}} \cite{Hai14}
and the related work, called {\it{paracontrolled calculus}}, due to Gubinelli, Imkeller and Perkowski \cite{GIP15},
there has arisen a renewed field 
of singular SPDEs, and now the 
$\Phi^{4}_{3}$-stochastic quantization equation is studied intensively by applying these
new methods (see e.g., \cite{CC18, MW17b, AK19, GH19} for recent developments on 
the $\Phi^{4}_{3}$-stochastic quantization equation).

In the present paper, we consider
a quantum field model in two-dimensional finite volume, which is different from
the $P(\Phi)_{2}$-model. 
This model also leads to interesting relativistic quantum fields, and 
was introduced by H\o egh-Krohn
\cite{Hoe71} in a Hamiltonian setting. Later its Euclidean version was constructed by Albeverio and 
H\o egh-Krohn \cite{AH74}.
In the latter paper, the {\it{${\rm{exp}}(\Phi)_{2}$-quantum field}}
\begin{equation}
\mu^{(\alpha)}_{\sf exp}(d\phi)=\frac{1}{Z^{(\alpha)}}
\exp \Big (
- \int_{\Lambda} \exp^{\diamond} ( \alpha \phi)(x) 
dx \Big )
\mu_{0}(d\phi)
\label{exp-Gibbs}
\end{equation}
was constructed and shown
to yield interesting relativistic quantum fields, where
$\Lambda={\mathbb T}^{2}$, 
$Z^{(\alpha)}>0$ is the normalizing constant,
$\alpha \in 
( -{\sqrt{4\pi}}, {\sqrt{4\pi}})$ is called the {\it{charge parameter}} and the {\it{Wick exponential}} 
$\exp^{\diamond} ( \alpha \phi)(x)$
is formally introduced by the expression
$$ \exp ^{\diamond} (\alpha \phi)(x)
\hspace{0.5mm}
=\exp \Big( \alpha \phi(x)-\frac{\alpha^{2}}{2} {\mathbb E}^{\mu_{0}} [ \phi(x)^{2}] \Big), \qquad x\in \Lambda.
$$
Here the diverging term ${\mathbb E}^{\mu_{0}} [ \phi(x)^{2}]$
plays a role of the Wick renormalization. Note that a random measure
$ \nu^{(\alpha)}_{\phi}(dx):= \exp ^{\diamond} (\alpha \phi)(x)dx$ 
on $\Lambda$ is called the {\it{Gaussian mulptiplicative chaos}}, which plays a central role in the theory of 
Liouville quantum gravity.
A connection between the ${\rm exp}(\Phi)_{2}$-quantum field model and problems
in representation theory of groups of mappings has been discussed in \cite{AHT81}. Recently,
the relevance of this model was rediscovered in connection 
with topics like Liouville quantum gravity and stochastic Ricci flow. See e.g., \cite{Kah85, RV14, DS11, DS19}
and references therein. 

The main purpose of the present paper is to study a parabolic SPDE 
\begin{equation}
\partial_{t}\Phi_{t}(x)=
\frac{1}{2}(\Delta-m_{0}^{2})\Phi_{t}(x)
-\frac{\alpha}{2} \exp^{\diamond} ( \alpha \Phi_{t}(x) )
+
{\dot W}_{t}(x), 
\qquad t>0,~x\in \Lambda,
\label{exp-SQE}
\end{equation}
which realizes the stochastic quantization for the 
${\rm{exp}}(\Phi)_{2}$-quantum field
$\mu^{(\alpha)}_{\sf exp}$. In the paper \cite{AR91} mentioned above, 
Albeverio and R\"ockner treated not only the $P(\Phi)_{2}$-case but also
the ${\rm{exp}}(\Phi)_{2}$-case, and they solved 
(\ref{exp-SQE}) weakly under $\vert \alpha \vert <{\sqrt{4\pi}}$
by using the Dirichlet form theory. 
By following the Girsanov transform approach in \cite{GG96}, 
Mihalache \cite{Mih06} constructed a unique probabilistically weak solution
to a modified SPDE 
\begin{equation}
\label{exp-modifySQE}
\begin{array}{rl}
\partial_{t}\Phi_{t}(x)=&{\displaystyle{
-\frac{1}{2}(m_{0}^{2}-\Delta)^{\varepsilon} \Phi_{t}(x)
-\frac{\alpha}{2}(m_{0}^{2}-\Delta)^{\varepsilon-1} \exp^{\diamond} ( \alpha \Phi_{t}(x) )}}
\\
&
{\displaystyle{
+(m_{0}^{2}-\Delta)^{\frac{\varepsilon-1}{2}}
{\dot W}_{t}(x), 
\qquad t>0,
~x\in \Lambda,}}
\end{array}
\end{equation}
under 
restrictive conditions on $0<\varepsilon<1$ and the charge parameter $\alpha$.
Strong uniqueness
of the generator of the modified SPDE
(\ref{exp-modifySQE}) was also discussed in \cite{AKMR19}. 
Nevertheless, to our best knowledges, there were few papers which study the original SPDE
(\ref{exp-SQE}).
Quite recently, influenced by the recent development of singular SPDEs,
Garban \cite{Gar18} studied (\ref{exp-SQE}) with $m_{0}=0$ (i.e., massless case).
Under a stronger condition 
than $\vert \alpha \vert<{\sqrt{4\pi}}$,
he constructed a unique strong solution to (\ref{exp-SQE}). 
See Remark \ref{Garban} below
for a detailed comparison with our results.
We should mention that elliptic SPDEs which also realize
the $\exp (\Phi )_2$-quantum field model were studied in e.g., \cite{AY02, ADG19}.
We further note that a much deeper analysis is possible for 
the ${\rm exp}(\Phi)_{1}$-stochastic quantization equation 
(i.e., the one-dimensional version of (\ref{exp-SQE})) because 
of non-necessity of renormalization. Uniqueness of both
the strong solution to the ${\rm exp}(\Phi)_{1}$-stochastic quantization equation 
and the corresponding generator have been proven in \cite{AKR12}.

In the present paper, 
we construct the time-global and pathwise-unique solution to the original SPDE 
\eqref{exp-SQE} by the
Da Prato-Debussche argument under $\vert \alpha \vert <\sqrt{4\pi}$. (It is easy to see that our argument in the present paper also works in the case of the modified 
SPDE (\ref{exp-modifySQE}). We omit it.) The key idea is that we regard the Wick exponentials of the Ornstein-Uhlenbeck process 
as an $L^2$-function in time and construct estimates.
The Wick exponentials of the Ornstein-Uhlenbeck process appears as an input of the solution map to the shifted equation.
To apply regularity structures or paracontrolled calculus we usually assume that the inputs are $B_{\infty ,\infty}^s$-valued processes.
However, the Wick exponentials of the Ornstein-Uhlenbeck process does not satisfy the condition.
Moreover, the nonlinear term of \eqref{exp-SQE} has exponential growth.
Hence, the SPDE \eqref{exp-SQE} is out of results by the general theories.
We do not construct any contraction map for the existence and uniqueness of the solution, but just prepare some estimates, directly construct the time-global solution and obtain the pathwise uniqueness.
By the uniqueness we also have the identification of the solution with the limit of the solutions to the stochastic quantization equations generated by the approximating measures to the 
$\exp (\Phi )_2$-measure, and with the process obtained by the Dirichlet form approach in \cite{AR91}.
The detail of the results are stated in Section \ref{sec:mainthm}.

Before closing this subsection, we would like to emphasize that
the $\exp (\Phi )_2$-quantum field model can be regarded as a model interpolating between the $\Phi ^4_2$-model and the $\Phi ^4_3$-model in the following sense. 
When we consider the shifted equation of 
\eqref{exp-SQE},
then the Wick exponential of the Ornstein-Uhlenbeck process appears as 
a coefficient and it is a $W^{-\alpha ^2/4\pi -\varepsilon ,2}$-valued process. 
On the other hand, in the case of the $\Phi^4_d$-model, the second-order Wick polynomial 
of the Ornstein-Uhlenbeck process appears as the most singular coefficient and it is a 
$W^{-d+2-\varepsilon ,\infty }$-valued process. By comparing the singularities of the coefficient, 
we have a relation $-\alpha ^2/4\pi = -d+2$. In view of the relation, formally $\alpha = \sqrt{4\pi}$ 
in the $\exp (\Phi )_2$-model associated to the $\Phi ^4_3$-model and $\alpha =0$ associated to 
the $\Phi ^4_2$-model. We remark that the relation is only based on the singularities and 
the integrability is ignored. The 
{\it{sine-Gordon quantum field model}}, which was studied in e.g., \cite{Fro76, FP77}, and recently, 
its stochastic quantization equation was also studied in \cite{HS16, CHS18} by applying regularity structures.
In terms of the singularities, the stochastic quantization equation for
the sine-Gordon model is same as one of the $\exp (\Phi)_{2}$-model.
However, sine functions are bounded and have bounded derivatives, while exponential functions are unbounded 
and the derivatives are also unbounded. This is another reason why we study the stochastic 
quantization of exponential models 
in the present paper by using a completely 
different argument from the one by \cite{HS16, CHS18}.
\subsection{Settings and main theorems}\label{sec:mainthm}
We begin with introducing some notations and objects. 
Throughout the paper, we fix $m_{0}=1$ for the simplicity of notation.
Let $\Lambda=\mathbb{T}^2=(\mathbb{R}/2\pi\mathbb{Z})^2$ be the two-dimensional torus
equipped the Lebesgue measure $dx$. Let $L^{2}(\Lambda; {\mathbb K})$ 
(${\mathbb K}={\mathbb R}, {\mathbb C}$) be the Hilbert space consisting all $\mathbb K$-valued
Lebesgue square integrable functions equipped with the usual inner product
$$ \langle f,g \rangle=\int_{\Lambda} f(x){\overline{g(x)}}dx, \quad f,g\in L^{2}(\Lambda; {\mathbb K}).$$
For ${k}=(k_{1}, k_{2}) \in {\mathbb Z}^{2}$ and $x=(x_{1}, x_{2}) \in \Lambda$,
we write $\vert k \vert=(k_{1}^{2}+k_{2}^{2})^{1/2}$ and
$k\cdot x=k_{1}x_{1}+k_{2}x_{2}$.
Since $C^{\infty}(\Lambda; {\mathbb K}) \subset L^{2}(\Lambda; {\mathbb K}) \subset {\mathcal D}'(\Lambda;\mathbb K)$,
the $L^{2}$-inner product $\langle \cdot, \cdot \rangle$ is naturally extended to the pairing of $C^{\infty}(\Lambda; \mathbb K)$ and its dual space ${\mathcal D}'(\Lambda;\mathbb K)$.
Let 
$\{ 
e_{k}; k \in \mathbb Z^{2} \}$ be the usual complete 
orthonormal system (CONS) of $L^{2}(\Lambda; \mathbb R)$
consisting of
$e_{(0,0)}(x)=(2\pi)^{-1}$ and
\begin{equation*}
%
e_{k}(x)
=
\frac{1}{{\sqrt 2}\pi}
\begin{cases}
\displaystyle{
\cos(k\cdot x)
},
& \text{ $k\in {\mathbb Z}^{2}_{+}$}
\\ 
\displaystyle{
\sin(k\cdot x)
},
& \text{ $k\in {\mathbb Z}^{2}_{-}$},
\end{cases}
\label{CONS}
\end{equation*}
where ${\mathbb Z}^{2}_{+}=\{ (k_{1}, k_{2}) \in {\mathbb Z}^{2} \vert \hspace{0.5mm} k_{1}>0 \} 
\cup \{ (0, k_{2}) \vert \hspace{0.5mm} k_{2}>0 \}$ and ${\mathbb Z}^{2}_{-}=-{\mathbb Z}^{2}_{+}$.
Although we work in the framework of real-valued functions, it is sometimes easier to do computations by using
the corresponding complex basis
$$
{\bf e}_{k}(x)=\frac{1}{2\pi}e^{{\sqrt{-1}}k\cdot x}, \quad  k\in {\mathbb Z^{2}},~x\in \Lambda. 
$$
For $s\in {\mathbb R}$, we define the Sobolev space of order $s$ with periodic boundary condition
by
$$ 
H^{s}=H^{s}(\Lambda)=\left\{ u\in {\mathcal D}'(\Lambda; {\mathbb R}) ; \sum_{k \in {\mathbb Z}^{2}} (1+\vert k \vert^{2})^{s} \vert \langle u, {\bf e}_{k} \rangle \vert^{2}<\infty
\right\},
$$
This space is a Hilbert space equipped with the inner product
$$ { (}u,v{ )}_{H^{s}}:=
\sum_{k \in {\mathbb Z}^{2}} (1+\vert k \vert^{2})^{s}   
\langle u, {\bf e}_{k} \rangle
{\overline{\langle v, {\bf e}_{k} \rangle}}   , 
\qquad u,v \in H^{s}.$$

We define the massive Gaussian free field measure $\mu_{0}$ by the centered Gaussian measure 
on $\mathcal{D}'(\Lambda)$ with covariance $(1-\triangle)^{-1}$, that is, determined by the formula
$$
\int_{\mathcal{D}'(\Lambda)}\langle \phi,{\bf e}_k\rangle \overline{\langle \phi,{\bf e}_{\ell}\rangle}\mu_0(d\phi)
=(1+|k|^2)^{-1}\mathbf{1}_{k=\ell},\quad
k,\ell\in\mathbb{Z}^2,
$$
where $\Delta$ is the Laplacian acting on $L^{2}(\Lambda)$ with periodic boundary condition.
Note that this formula implies
$$ \int_{\mathcal{D}'(\Lambda)}\|\phi\|_{H^{-\varepsilon}}^2\mu_0(d\phi)<\infty, \quad
\varepsilon>0, $$
and thus the Gaussian free field measure $\mu_0$ has a full support on $H^{-\varepsilon}(\Lambda)$.
For a charge parameter $\alpha \in (-{\sqrt{4\pi}}, {\sqrt{4\pi}})$,
we then define the \emph{$\exp (\Phi) _2$-quantum field} (or the \emph{$\exp (\Phi) _2$-measure}) 
$\mu^{(\alpha)}=\mu_{\sf{exp}}^{(\alpha)}$ on $\mathcal{D}'(\Lambda)$ by
\begin{equation}\label{expphimeas}
\mu^{(\alpha)}(d\phi)=\mu_{\sf{exp}}^{(\alpha)}(d\phi):=\frac{1}{Z^{(\alpha)}} \exp \left ( - \int_\Lambda\exp^\diamond(\alpha \phi )(x)dx\right ) \mu _0 ( d\phi ),
\end{equation}
where
$Z^{(\alpha)}>0$ is the normalizing constant and
$\exp^\diamond(\alpha\cdot)$ is the \emph{Wick exponential} which will be rigorously
constructed 
in Section \ref{sec:OU}. 
Since we prove in Section \ref{sec:OU} that the function $\int_{\Lambda} 
\exp^\diamond(\alpha \phi )(x)dx$ is a positive $L^{2}(\mu_{0})$-function for all $\vert \alpha \vert
<{\sqrt{4\pi}}$, we may also regard 
$\mu^{(\alpha)}$ as a probability measure on $H^{-\varepsilon}(\Lambda)$ (see Corollary \ref{cor:expmeas}).

In the present paper, we consider a \emph{stochastic quantization equation} 
associated with $\exp(\Phi)_2$-measure, that is, a parabolic SPDE given by
\begin{equation}\label{expsqe}
\partial _t \Phi_t (x) = \frac 12 (\triangle-1) \Phi_t (x) - \frac\alpha2 \exp^\diamond(\alpha \Phi_t)(x) + \dot W_t (x) , \quad t>0,\quad x\in \Lambda,
\end{equation}
where $W=\{ W_t(x); t\geq 0, x\in \Lambda \}$ is an $L^{2}(\Lambda)$-cylindrical Brownian motion 
defined on a filtered probability space $(\Omega,\mathcal{F},({\mathcal F}_{t})_{t\geq 0},
\mathbb{P})$. This driving noise is defined 
by the following convenient Fourier series representation
$$ W_{t}(x)=\sum_{k\in {\mathbb Z}^{2}} w^{(k)}_{t} e_{k}(x), \quad t\geq 0, x\in \Lambda,$$
where $\{w^{(k)} \}_{k\in {\mathbb Z}^{2}}$ is a sequence of independent 
one-dimensional $({\mathcal F}_{t})_{t\geq 0}$-Brownian motions starting at $0$.
See \cite[Chapter 4]{DZ92} for details.
The Wick exponential $\exp ^\diamond$ is defined only on almost everywhere with respect to suitable Gaussian measures with supports in distributions and ill-defined for a general distribution.
Hence, the exponential term of the SPDE \eqref{expsqe} is difficult to treat as it is, because $\Phi_t$ takes values in $\mathcal{D}'(\Lambda) \setminus C(\Lambda)$.
For this reason, we first consider 
an approximating equation given by regularizing the white noise $\dot{W}_t$.
Let $\psi$ be a Borel function on ${\mathbb R}^2$ with the following properties.
\begin{itemize}
\item $0\le\psi(x)\le1$ for any $x\in\mathbb{R}^2$.
\item $\psi(x)=\psi(-x)$ for any $x\in\mathbb{R}^2$.
\item $\sup_{x\in\mathbb{R}^2\setminus \{ 0\}} |x|^{-\theta}|\psi (x) -1| <\infty$ for some $\theta\in(0,1)$.
\item $\sup_{x\in\mathbb{R}^2} |x|^m |\psi (x)|<\infty$ for some $m\geq 4$.
\end{itemize}
Note that $\psi$ need not be continuous at $x\neq0$. For example, an indicator function $\psi=\mathbf{1}_K$ is allowed, if $K\subset\mathbb{R}^2$ is compact, $K=-K:=\{-x;x\in K\}$, and $0$ is an interior point of $K$.
For such a cut-off function $\psi$, we define an operator $P_N$ on ${\mathcal D}'(\Lambda )$ by
$$
P_N f (x) = \sum _{k\in {\mathbb Z}^2} \psi (2^{-N} k) \langle f,{\bf{e}}_k \rangle {\bf{e}}_k (x), \quad N\in \mathbb N,~x\in \Lambda.
$$
By the assumption on $\psi$, it is easy to show the following properties.
\begin{itemize}
\item $\|P_Nf\|_{H^{2}}\lesssim2^{4N}\|f\|_{H^{-2}}$. In particular, $P_Nf\in C(\Lambda)$ for $f\in H^{-2}(\Lambda)$ by the Sobolev embedding.
\item $\displaystyle{\lim_{N\to\infty}\|P_Nf-f\|_{H^s}=0}$ for $f\in H^s(\Lambda)$.
\end{itemize}

Then the first result is stated as follows.
\begin{thm}\label{mainthm1}
Let $|\alpha|<\sqrt{4\pi}$ and $\varepsilon>0$.
Let $N\in\mathbb{N}$ and consider the initial value problem
\begin{equation}\label{expsqe1}
\left\{
\begin{aligned}
\partial _t \Phi_t^N &= \frac 12 (\triangle-1) \Phi_t^N 
- \frac\alpha2 \exp\left(\alpha \Phi_t^N-\frac{\alpha^2}2 C_N\right) + P_N\dot W_t ,\\
\Phi_0^N&=P_N\phi,
\end{aligned}
\right.
\end{equation}
where $\phi\in\mathcal{D}'(\Lambda)$ and
$$
C_N:=\frac1{4\pi ^2}\sum_{k\in\mathbb{Z}^2}\frac{\psi(2^{-N}k)^2}{1+|k|^2}.
$$
Then for $\mu_0$-a.e. $\phi\in\mathcal{D}'(\Lambda)$, the unique time-global classical solution $\Phi^N$ converges to an $H^{-\varepsilon}$-valued stochastic process $\Phi$ in the space $C([0,T],H^{-\varepsilon}(\Lambda))$ 
for any $T>0$ $\mathbb P$-almost surely. Moreover, the limit $\Phi$ is independent to the choice of $\psi$.
\end{thm}
We call the ${\Phi}$ obtained in Theorem \ref{mainthm1} the \emph{strong solution} of the SPDE \eqref{expsqe} with the initial value $\phi$.

\begin{rem}
Since the $\exp(\Phi)_2$-measure $\mu^{(\alpha)}$ is absolutely continuous with respect to $\mu_0$ 
under $|\alpha|<\sqrt{4\pi}$ (see Corollary \ref{cor:expmeas}), 
the phrase ``$\mu_0$-a.e. $\phi$" can be replaced by ``$\mu^{(\alpha)}$-a.e. $\phi$".
\end{rem}

As another approach to the SPDE \eqref{expsqe} we consider the regularized ${\rm exp}(\Phi)_{2}$-measure 
$\mu^{(\alpha)}$ by
\begin{align}\label{expNmeas}
\mu_N^{(\alpha)}(d\phi):=\frac{1}{Z^{(\alpha)}_N} \exp \left\{ - \int_\Lambda\exp\left(\alpha P_N\phi (x)-\frac{\alpha^2}{2}C_N\right)dx\right\} \mu _0 ( d\phi ),
\quad N\in\mathbb{N},
\end{align}
where $Z^{(\alpha)}_N>0$ is the normalizing constant, and the SPDE associated with this measure.
The sequence $\{\mu_N^{(\alpha)}\}$ of probability measures weakly converges to 
$\mu^{(\alpha)}$ (see Corollary \ref{cor:expmeas}).
Let $\rho$ be a nonnegative function on $\mathbb{R}^2$ and let
$$
P_Nf(x)=\int_{\mathbb{R}^2}2^{2N}\rho(2^N(x-y))\tilde{f}(y)dy,\quad x\in {\Lambda}, \ f\in\mathcal{D}'(\Lambda),
$$
where $\tilde{f}$ is the periodic extension of $f$ to $\mathbb{R}^2$.
Then the operator $P_N$ is a nonnegative operator, i.e. $P_N f \ge0$ if $f\ge0$.
Denote the Fourier transform $\rho$ by $\psi$ and assume that $\psi$ satisfies the conditions above.
We remark that we are able to choose usual mollifiers as $P_N$.
Indeed, if $\rho$ is a nonnegative and radial function in the Schwartz class with $\int _{\mathbb R ^2}\rho (x) dx =1$, then its Fourier transform $\psi$ is also in the Schwartz class and hence satisfies the conditions of $\psi$ above.

Then the second result is stated as follows.
\begin{thm}\label{mainthm2}
Let $|\alpha|<\sqrt{4\pi}$, $\varepsilon>0$, and $P_N$ as above.
Let $N\in\mathbb{N}$ and consider the solution ${\bf\Phi}^N={\bf\Phi}^N(\phi)$ of an SPDE
\begin{equation}\label{expsqe2}
\left\{
\begin{aligned}
\partial_t{\bf\Phi}_t^N&=\frac12(\triangle-1){\bf\Phi}_t^N
-\frac\alpha2 P_N\exp\left(\alpha P_N{\bf\Phi}_t^N-\frac{\alpha^2}2C_N\right)+\dot{W}_t,\\
{\bf\Phi}_0^N&=\phi\in\mathcal{D}'(\Lambda).
\end{aligned}
\right.
\end{equation}
Let $\xi_N$ be a random variable with the law $\mu_N^{(\alpha)}$ and independent of $W$.
Then $\bar{\bf\Phi}^N={\bf\Phi}(\xi_N)$ is a stationary process and the family $\{\bar{\bf\Phi}^N\}_{N=1}^{\infty}$ 
converges in law to the strong solution $\bar{\bf\Phi}$ of \eqref{expsqe} with an initial law $\mu^{(\alpha)}$, in the space 
$C([0,T];H^{-\varepsilon}(\Lambda))$ for any $T>0$.
Moreover, the law of $\bar{\bf\Phi}_t$ is $\mu^{(\alpha)}$ for any $t\ge0$.
\end{thm}

Now we are in a position to introduce a
pre-Dirichlet form $({\cal E},{\mathfrak F}C_{b}^{\infty})$.
We fix $\beta \in(\frac{\alpha^2}{4\pi},1)$ and set 
$H=L^{2}(\Lambda; {\mathbb R})$ and $E=H^{-\beta}(\Lambda)$. 
Let ${\mathfrak F}C_{b}^{\infty}$ be the space of all smooth cylinder functions
on $E$ having the form
$$
F(\phi)=f(\langle \phi, l_{1} \rangle , \ldots , \langle \phi, l_{n} \rangle), \quad \phi \in E,
$$
with $n\in {\mathbb N}$, $f
\in C^{\infty}_{b}({\mathbb R}^{n}; {\mathbb R})$ 
and $l_{1}, \ldots , l_{n} \in {\rm Span}\{ e_{k}; k\in {\mathbb Z}^{2} \}$.
Since we have supp$(\mu^{(\alpha)})=E$, two different functions in
${\mathfrak F}C_{b}^{\infty}(K)$ are also different in $L^{p}(\mu^{(\alpha)})$-sense.
Note that ${\mathfrak F}C_{b}^{\infty}$ is dense in $L^{p}(\mu^{(\alpha)})$ for all $p\geq 1$.
For $F \in {\mathfrak F}C_{b}^{\infty}$, we define the $H$-Fr\'echet derivative 
$D_{H}F:E\to H$ by
$$
D_{H}F(\phi):=\sum_{j=1}^{n}{\partial_{j}} {f}
\big(
\langle \phi,l_{1} \rangle , \ldots,
\langle \phi,l_{n} \rangle
\big)l_{j}, \quad \phi\in E.
$$
We consider a pre-Dirichlet form $({\cal E},{\mathfrak F}C_{b}^{\infty})$
which is given by 
\begin{equation}
{\cal E}(F,G)=
\frac{1}{2} 
\int_{E} \big( D_{H}F(w), D_{H}G(w) \big)_{H}
 \mu ^{(\alpha )}(dw),~~F,G\in {\mathfrak F}C_{b}^{\infty}.
 \label{DF-intro}
\end{equation}
By following the argument in \cite{AR91, AKMR19}, we easily deduce that
$({\cal E},{\mathfrak F}C_{b}^{\infty})$ is closable on $L^{2}(\mu^{(\alpha)})$.
So we can define ${\cal D(E)}$ as the completion of ${\mathfrak F}C_{b}^{\infty}$
with respect to ${\cal E}_{1}^{1/2}$-norm. Thus, by directly applying the general methods in the theory of Dirichlet forms (cf. \cite{MR92, CF12}), we can prove quasi-regularity of $({\cal E}, {\cal D(E)})$ and the existence of a diffusion process 
${\mathbb M}=(\Theta, {\mathcal G}, ( {\mathcal G}_{t})_{t\geq 0},  (\Psi_{t})_{t\geq 0}, ( {\mathbb Q}_{\phi} )_{\phi \in E})$
properly associated with $({\cal E}, {\cal D(E)})$.

The following theorem says that the diffusion process
$\Psi=(\Psi_{t})_{t\geq 0}$
coincides with the strong solution $\Phi$.
\begin{thm}\label{mainthm3}
Let $|\alpha|<\sqrt{4\pi}$. Then for $\mu^{(\alpha)}$-a.e. $\phi$, the diffusion process
$\Psi$ coincides with the strong solution $\Phi$ of the SPDE \eqref{expsqe} 
driven by some 
$L^{2}(\Lambda)$-cylindrical 
$({\mathcal G}_{t})$-Brownian motion ${\mathcal W}=({\mathcal W}_{t})_{t\geq 0}$
with the initial value $\phi$, ${\mathbb Q}_\phi$-almost surely.
\end{thm}
\begin{rem} \label{Garban}
Garban \cite{Gar18} studied the following SPDE for a parameter $\gamma\in(0,2)$.
\begin{align}\label{eq:Garban}
\partial_tX_t(x)=\frac{1}{4\pi}\triangle X_t(x)-e^{\gamma X_t(x)}+\dot{W}_t(x),\quad
t>0,\quad x\in(\mathbb{R}/\mathbb{Z})^2.
\end{align}
This equation is essentially the same as \eqref{expsqe} for the existence and uniqueness of the solution.
(The massless version of the $\exp(\Phi)_2$-measure is called a \emph{Liouville measure} in \cite{Gar18}.)
To see this, we consider the equation \eqref{eq:Garban} in the whole plane $\mathbb{R}^2$.
After that, by setting $\tilde{X}_t(x)=\sqrt{2\pi}^{-1}X_t(\sqrt{2\pi}^{-1}x)$, we have the equation
$$
\partial_t\tilde{X}_t(x)=\frac{1}{2}\triangle \tilde{X}_t(x)-\sqrt{2\pi}^{-1}e^{\sqrt{2\pi}\gamma \tilde{X}_t(x)}+\dot{\tilde{W}}_t(x),
$$
where $\tilde{W}_t(x) := \sqrt{2\pi}^{-1}W_t(\sqrt{2\pi}^{-1}x)$ has the same law as $W_t(x)$.
Therefore the relation between two parameters $\alpha$ and $\gamma$ should be given by
$$
\alpha=\sqrt{2\pi}\gamma.
$$
Garban \cite{Gar18} obtained the local well-posedness of \eqref{eq:Garban} when $\gamma<2\sqrt{2}-\sqrt{6}\fallingdotseq0.38$, and constructed the strong solution locally in time without continuity in $W$ like Theorem \ref{mainthm1} when $\gamma<2\sqrt{2}-2\fallingdotseq0.83$.
We remark that, if we multiply the term $e^{\gamma X_t(x)}$ in (\ref{eq:Garban}) by $\gamma$, the equation will be same as (\ref{expsqe}) up to multiplications by absolute constants and changes of notations.
In this paper, we obtain both the global well-posedness and the continuity in $W$ in larger region $\gamma<\sqrt{2}\fallingdotseq 1.41$, which comes from the assumption $\alpha<\sqrt{4\pi}$ and the relation $\alpha=\sqrt{2\pi}\gamma$.
The difference in the proofs is that we directly construct the global solution by the compact embedding theorem and checked the uniqueness independently, while the fixed point theorem is applied in \cite{Gar18}.

For the existence of the invariant measure, \eqref{expsqe} and \eqref{eq:Garban} are different, because the $0$-Fourier mode $\langle X_t,  {\bf{e}}_{(0,0)} \rangle$ does not have dissipativity. This implies that in the case without mass, the invariant measure is an infinite measure and we need extra treatments, for e.g. by punctures (See \cite[Theorem 1.12]{Gar18}).
\end{rem}

\begin{rem}
The assumption $|\alpha|<\sqrt{4\pi}$ is necessary to discuss the problem in $L^2$-regime with respect to the probability measure.
Indeed, we apply the fact that the Wick exponential of the Ornstein-Uhlenbeck process, which appears as the noise term in the shifted equation, belongs to the $L^2$-space with respect to the time parameter almost surely.
On the other hand, the Wick exponential with respect to the free field measure has been constructed for $\alpha \in (-\sqrt{8\pi}, \sqrt{8\pi})$.
However, if $\sqrt{4\pi} < |\alpha | < \sqrt{8\pi}$, the constructed Wick exponential is not in $L^2$-space, but in $L^p$-space for some $p\in (1,2)$ with respect to the free field measure. 
See \cite{Kus92}, for the detail.
\end{rem}

The organization of the rest of the present paper is as follows.
In Section \ref{sec:OU} we introduce the exponential Wick product on the Gaussian free field measure and study the regularity of the Wick exponentials and the quantum field generated by them.
Furthermore, we also introduce the process generated by the Wick exponentials of the Orinstein-Uhlenbeck process 
and see the stability of the process in the initial value.
In Section \ref{sec:wellposed} we prove Theorem \ref{mainthm1}.
Precisely, we will see the existence and uniqueness of the solution to the shifted equation.
We remark that the argument in Section \ref{sec:wellposed} is pathwise and that we directly construct a solution global in time.
The key idea of the proof is that we regard the Wick exponentials of the Orinstein-Uhlenbeck process as an $L^2$-function in time.
In the section we also discuss some functional inequalities of nonnegative distributions (see Section \ref{sec:nnegdist}).
In Section \ref{sec:stationary} we prepare a sequence of stationary solutions associated to the approximating measures of $\exp (\Phi )_2$-measures and see the convergence of the sequence to the solution obtained in Theorem \ref{mainthm1} (Theorem \ref{mainthm2}).
The stability of the Wick exponentials of the Orinstein-Uhlenbeck process 
obtained in Section \ref{sec:OU} is applied in the proof.
In Section \ref{sec:DF} we prove Theorem \ref{mainthm3}, which concludes that the process constructed by 
Dirichlet forms coincides with the solution obtained in Theorem \ref{mainthm1}.
In particular, it yields the pathwise uniqueness of the SPDE associated to the Dirichlet form.

Throughout this paper, we use the notation $A\lesssim B$ for two functions $A = A(\lambda)$ and $B = B(\lambda)$ of a variable $\lambda$, if there exists a constant $c>0$ independent of $\lambda$ such that $A\le cB$. We write $A\simeq B$ if $A\lesssim B$ and $B\lesssim A$. We write $A\lesssim_\mu B$ if we want to emphasize that the constant $c$ depends on another variable $\mu$.


\section{Wick exponential of the Ornstein-Uhlenbeck process}\label{sec:OU}

In this section, we prepare
some properties of 
the infinite-dimensional Ornstein-Uhlenbeck (OU in short) process, and Wick exponentials.

\subsection{Infinite-dimensional OU process}
Let $X=X(\phi)$ be the unique solution of the initial value problem
\begin{align}\label{eq:OU}
\left\{
\begin{aligned}
\partial_tX_t&=\frac12(\triangle-1)X_t+\dot{W}_t,\\
X_0&=\phi
\end{aligned}
\right.
\end{align}
for $\phi\in\mathcal{D}'(\Lambda)$.
It is known that $\mu_0$ is an invariant measure of the OU process $X$
(see e.g., \cite[Theorem 6.2.1]{DZ96}).

\begin{prop}\label{prop:aprioriOU}
For $\varepsilon>0$, $\delta\in(0,1)$, and $m\in\mathbb{N}$, there exists a constant $C>0$ such that one has the a priori estimate
\begin{align}\label{ineq:aprioriOU}
\mathbb{E}{\Big [}  \|X(\phi)\|_{C([0,\infty);H^{-\varepsilon})\cap C^{\delta/2}([0,\infty);H^{-\varepsilon-\delta})}^m {\Big ]}
\le C(1+\|\phi\|_{H^{-\varepsilon}}^m).
\end{align}
\end{prop}

\begin{proof}
$X$ solves \eqref{eq:OU} in the mild form
\begin{align}\label{eq:OUmild}
X_t=e^{\frac12(\triangle-1)t}\phi+\int_0^te^{\frac12(\triangle-1)(t-s)}dW_s
=:X_t^{(1)}+X_t^{(2)}, \quad t\geq 0.
\end{align}
For $X^{(1)}$, \eqref{ineq:aprioriOU} is a consequence of Proposition \ref{prop:heatsemigr}.
The continuity of $t\mapsto X_t^{(1)}$ in $H^{-\varepsilon}$ follows from the dominated convergence theorem.
For $X^{(2)}$, by the It\^o isometry,
\begin{align*}
\mathbb{E} \big [ \|X_t^{(2)}\|_{H^{-\varepsilon}}^2 \big ]
=\sum_{k\in\mathbb{Z}^2}(1+|k|^2)^{-\varepsilon}\int_0^te^{-(1+|k|^2)(t-s)}ds
\le\sum_{k\in\mathbb{Z}^2}(1+|k|^2)^{-1-\varepsilon}<\infty.
\end{align*}
Let $0\le s<t\le T$. By the semigroup property,
\begin{align*}
X_t^{(2)}-X_s^{(2)}=\left(e^{\frac12(\triangle-1)(t-s)}-1\right)X_s^{(2)}+\int_s^te^{\frac12(\triangle-1)(t-r)}dW_r.
\end{align*}
By the It\^o isometry again,
\begin{align*}
\mathbb{E} \Big [ \|X_t^{(2)}-X_s^{(2)}\|_{H^{-\varepsilon-\delta}}^2 \Big ]
&\lesssim(t-s)^{\delta}\mathbb{E} \Big [ \|X_s^{(2)}\|_{H^{-\varepsilon}}^2 \Big ]
+\sum_{k\in\mathbb{Z}^2}(1+|k|^2)^{-\varepsilon-\delta}\frac{1-e^{-(1+|k|^2)(t-s)}}{1+|k|^2}\\
&\lesssim(t-s)^{\delta}\mathbb{E} \Big [ \|X_s^{(2)} \|_{H^{-\varepsilon}}^2 \Big ]
+\sum_{k\in\mathbb{Z}^2}(1+|k|^2)^{-\varepsilon-\delta}(1+|k|^2)^{-1+\delta}(t-s)^{\delta}\\
&\lesssim|t-s|^{\delta}.
\end{align*}
By the hypercontractivity of Gaussian random variables, we have 
$$\mathbb{E}\Big [\|X_t^{(2)}-X_s^{(2)}\|_{H^{-\varepsilon-\delta}}^{2m} \Big ]
\le C_m|t-s|^{\delta m}, \quad m\in\mathbb{N}$$ 
for some $C_m>0$. Hence \eqref{ineq:aprioriOU} is a consequence of the Kolmogorov's theorem.
\end{proof}

\subsection{Wick exponential of GFF}

For $x\in\mathbb{R}$ and $\sigma\ge0$, let $\{H_n(x;\sigma)\}_{n=0}^\infty$ be the Hermite polynomials defined via the generating function
$$
e^{\alpha x-\frac{\alpha^2}2\sigma}=\sum_{n=0}^\infty\frac{\alpha^n}{n!}H_n(x;\sigma),
\quad\alpha\in\mathbb{R}.
$$
It is well known that, if $X$ and $Y$ are jointly Gaussian random variables with means $0$ and covariances $\sigma_X$ and $\sigma_Y$ respectively, then one has
\begin{align}\label{eq:hermiteOG}
\mathbb{E}\left[H_n(X;\sigma_X)H_m(Y;\sigma_Y)\right]=\delta_{nm}n!\mathbb{E}[XY]^n.
\end{align}

Let $\phi$ be a generic element of the probability space $(\mathcal{D}'(\Lambda),\mu_0)$. Since $\mu _0$-a.e. $\phi\in H^{-\varepsilon}$, the Wick exponential of $\phi$ is defined via an approximation.
Recall that $P_N$ is an operator on ${\mathcal D}'(\Lambda )$ defined by
$$
P_N f (x) = \sum _{k\in {\mathbb Z}^2} \psi (2^{-N} k) \langle f,{\bf e}_k \rangle {\bf e}_k (x).
$$
For simplicity, denote $\psi (2^{-N}\cdot )$ by $\psi _N$. 
We define the approximating Wick exponential $\exp_N^\diamond (\alpha \phi )$ by
$$
\exp_N^\diamond(\alpha\phi)(x):=\sum_{n=0}^\infty\frac{\alpha^n}{n!}H_n \big( P_N\phi (x);C_N \big),
\quad x\in \Lambda,
$$
where
$$
C_N=\int_{\mathcal{D}'(\Lambda)}(P_N\phi(x))^2\mu_0(d\phi)=\frac1{4\pi ^2}\sum_{k\in\mathbb{Z}^2}\frac{\psi_N(k)^2}{1+|k|^2}.
$$
The fact that $\exp_N^\diamond(\alpha\cdot)\ge0$ is obvious, because
$$
\exp_N^\diamond(\alpha \phi)(x)=\exp\left(\alpha P_N\phi(x) -\frac{\alpha^2}2C_N\right)\ge0, \quad x\in \Lambda.
$$

\begin{thm}\label{thm:expgff}
Let $|\alpha|<\sqrt{4\pi}$ and $\beta\in(\frac{\alpha^2}{4\pi},1)$.
Then the sequence of functions $\{\exp_N^\diamond(\alpha\phi)\}$ converges in $H^{-\beta}$, $\mu_0$-almost everywhere and in $L^2(\mu_0; H^{-\beta})$.
Moreover, the limit $\exp^\diamond(\alpha\phi)$ is independent of the choice of $\psi$.
\end{thm}

\begin{proof}
In the proof, all constants $C$ used below depend neither on $n$ nor $N$.
Let $N\in {\mathbb N}$ and $\ell\in\mathbb{Z}^2$. By the formula \eqref{eq:hermiteOG} and the fact that $H_0(x;\sigma)=1$, we have
\begin{align*}
&\int_{\mathcal{D}'(\Lambda)}\left|\left\langle \exp_{N+1}^\diamond(\alpha\phi) - \exp_N^\diamond(\alpha\phi) ,
{\bf e}_\ell \right\rangle\right|^2\mu_0(d\phi)\\
&= \sum _{n=1}^\infty \frac{\alpha^{2n}}{(n!)^2}
\int_{\mathcal{D}'(\Lambda)}
\left| \left\langle H_n(P_{N+1}\phi;C_N) - H_n(P_N\phi;C_N) , {\bf e}_\ell \right\rangle \right|^2\mu_0(d\phi)\\
&=\sum _{n=1}^\infty \frac{\alpha^{2n}}{(n!)^2} \int_\Lambda\int_\Lambda
\int_{\mathcal{D}'(\Lambda)}
\Bigl[ \left\{H_n(P_{N+1}\phi(x);C_N) - H_n(P_N\phi(x);C_N)\right\}\\
&\hspace{3cm}\times\left\{H_n(P_{N+1}\phi(y);C_N) - H_n(P_N\phi(y);C_N)\right\}\Bigr]\mu_0(d\phi)
\overline{{\bf e}_\ell(x)}{\bf e}_\ell(y)dxdy\\
&=
\sum _{n=1}^\infty \frac{\alpha^{2n}}{n!} \int _{\Lambda} \int _{\Lambda}
\left( \sum _{k\in {\mathbb Z}^2}\frac{\psi _{N+1}(k)^2}{1+|k|^2} {\bf e}_k (x) \overline{{\bf e}_k(y)}\right) ^n 
\overline{{\bf e}_\ell(x)} {\bf e}_\ell(y) dx dy \\
&\quad
-2 \sum _{n=1}^\infty \frac{\alpha^{2n}}{n!} {\rm Re} \int _{\Lambda} \int _{\Lambda}
\left( \sum _{k\in {\mathbb Z}^2}\frac{\psi _N(k) \psi _{N+1}(k)}{1+|k|^2} {\bf e}_k (x) 
\overline{{\bf e}_k(y)}\right) ^n \overline{{\bf e}_\ell(x)} {\bf e}_\ell(y) dx dy \\
&\quad
+ \sum _{n=1}^\infty \frac{\alpha^{2n}}{n!} \int _{\Lambda} \int _{\Lambda}
\left( \sum _{k\in {\mathbb Z}^2}\frac{\psi _N(k)^2}{1+|k|^2} {\bf e}_k (x) \overline{{\bf e}_k(y)}\right) ^n 
\overline{{\bf e}_\ell(x)} {\bf e}_\ell(y) dx dy \\
&=
\sum _{n=1}^\infty \frac{\alpha^{2n}}{(2\pi )^{n} n!}
\sum _{\substack{k_1,k_2,\dots, k_n \in {\mathbb Z}^2;\\ k_1+k_2+\dots +k_n =\ell}}
\frac{\psi _{N+1}(k_1)^2\psi _{N+1}(k_2)^2 \cdots \psi _{N+1}(k_n)^2}
{(1+|k_1|^2)(1+|k_2|^2) \cdots (1+|k_n|^2)}\\
&\quad
-2 \sum _{n=1}^\infty \frac{\alpha^{2n}}{(2\pi )^{n} n!}
\sum _{\substack{k_1,k_2,\dots, k_n \in {\mathbb Z}^2;\\ k_1+k_2+\dots +k_n =\ell}}
\frac{\psi _{N}(k_1)\psi _{N+1}(k_1) \psi _{N}(k_2)\psi _{N+1}(k_2)\cdots \psi _N(k_n) \psi _{N+1}(k_n)}{(1+|k_1|^2)(1+|k_2|^2) \cdots (1+|k_n|^2)} \\
&\quad
+ \sum _{n=1}^\infty \frac{\alpha^{2n}}{(2\pi )^{n} n!}
\sum _{\substack{k_1,k_2,\dots, k_n \in {\mathbb Z}^2;\\ k_1+k_2+\dots +k_n =\ell}}
\frac{\psi _N(k_1)^2\psi _N(k_2)^2 \cdots \psi _N(k_n)^2}{(1+|k_1|^2)(1+|k_2|^2) \cdots (1+|k_n|^2)} \\
&= \sum _{n=1}^\infty \frac{\alpha^{2n}}{(2\pi )^{n} n!}
\sum _{\substack{k_1,k_2,\dots, k_n \in {\mathbb Z}^2;\\ k_1+k_2+\dots +k_n =\ell}}
\frac{\left[ \psi _{N+1}(k_1) \psi _{N+1}(k_2) \cdots \psi _{N+1}(k_n) - \psi _N(k_1) \psi _N(k_2) \cdots \psi _N(k_n) \right] ^2}{(1+|k_1|^2)(1+|k_2|^2) \cdots (1+|k_n|^2)} .
\end{align*}
Since the assumptions on $\psi$ yields that, for any $\lambda \in(0,\theta)$,
\begin{align*}
&|\psi _{N+1}(k_1) \psi _{N+1}(k_2) \cdots \psi _{N+1}(k_n) - \psi _N(k_1) \psi _N(k_2) \cdots \psi _N(k_n)| \\
&\leq \sum _{j=1}^n \left( \left| \psi (2^{-N-1}k_j) - 1 \right| + \left| \psi (2^{-N}k_j) - 1 \right| \right) \\
&\leq C 2^{-\lambda N} \sum _{j=1}^n |k_j|^\lambda,
\end{align*}
hence we have
\begin{align*}
&\int_{\mathcal{D}'(\Lambda)}\left\| \exp_{N+1}^\diamond(\alpha\phi) - \exp_N^\diamond(\alpha\phi) \right\| _{H^{-\beta}}^2\mu_0(d\phi)\\
&\leq C 2^{-\lambda N} \sum _{n=1}^\infty \frac{\alpha^{2n}}{(2\pi )^{n} (n-1)!}
\sum _{\ell\in {\mathbb Z}^2} \frac{1}{(1+|\ell|^2)^\beta}
\sum _{\substack{k_1,k_2,\dots, k_n \in {\mathbb Z}^2;\\ k_1+k_2+\dots +k_n =\ell}}
|k_1|^\lambda \prod _{m=1}^n \frac{1}{1+|k_m|^2}\\
&\leq C 2^{-\lambda N} \sum _{n=1}^\infty \frac{\alpha^{2n}}{(2\pi )^{n} (n-1)!}
\sum _{\ell\in {\mathbb Z}^2} \frac{1}{(1+|\ell|^2)^\beta}
\sum _{\substack{k_1,k_2,\dots, k_n \in {\mathbb Z}^2;\\ k_1+k_2+\dots +k_n =\ell}}
\frac1{(1+|k_1|^2)^{1-\lambda}}\prod _{m=2}^n \frac{1}{1+|k_m|^2}.
\end{align*}
By the Young's inequality,
\begin{align*}
&
\left( \int_{\mathcal{D}'(\Lambda)} \left( \sum _{N=1}^\infty \left\| \exp_{N+1}^\diamond(\alpha\phi) - \exp_N^\diamond(\alpha\phi) \right\| _{H^{-\beta}} \right) ^2 \mu_0(d\phi) \right) ^{1/2}
\\
&\leq \sum _{N=1}^\infty \left( \int_{\mathcal{D}'(\Lambda)} \left\| \exp_{N+1}^\diamond(\alpha\phi) - \exp_N^\diamond(\alpha\phi) \right\| _{H^{-\beta}}^2 \mu_0(d\phi) \right) ^{1/2}
\\
&\leq \sum _{N=1}^\infty 2^{-\lambda N/2}+ \sum _{N=1}^\infty 2^{\lambda N/2}
\int_{\mathcal{D}'(\Lambda)}
\left\| \exp_{N+1}^\diamond(\alpha\phi) - \exp_N^\diamond(\alpha\phi) \right\| _{H^{-\beta}} ^2 \mu_0(d\phi).
\end{align*}
In view of this inequality, for the almost sure and $L^2$-convergence of $\{ \exp_N^\diamond(\alpha\phi) \}$ it is sufficient to show
\begin{equation}\label{ineq:sffcondwexp}
\sum _{n=1}^\infty \frac{\alpha^{2n}}{(2\pi )^{n} (n-1)!}
\sum _{\ell\in {\mathbb Z}^2} \frac{1}{(1+|\ell|^2)^\beta}
\sum _{\substack{k_1,k_2,\dots, k_n \in {\mathbb Z}^2;\\ k_1+k_2+\dots +k_n =\ell}}
\frac1{(1+|k_1|^2)^{1-\lambda}}\prod _{m=2}^n \frac{1}{1+|k_m|^2}<\infty
\end{equation}
for sufficiently small $\lambda >0$.
By using the Green function
\begin{align*}
K^\gamma (x,y)&:= \sum _{k\in {\mathbb Z}^2} (1+|k|^2) ^{-\gamma} {\bf e}_k(x) \overline{{\bf e}_k(y)},
\end{align*}
of $(1-\triangle)^\gamma$ for $\gamma\in(0,1]$, we have
\begin{align*}
&\sum _{\ell\in {\mathbb Z}^2} \frac{1}{(1+|\ell|^2)^\beta}
\sum _{\substack{k_1,k_2,\dots, k_n \in {\mathbb Z}^2;\\ k_1+k_2+\dots +k_n =\ell}}
\frac1{(1+|k_1|^2)^{1-\lambda}}\prod _{m=2}^n \frac{1}{1+|k_m|^2}\\
&=(2\pi )^n \sum _{\ell\in {\mathbb Z}^2} \int _{\Lambda} \int _{\Lambda} \left( \sum _{k\in {\mathbb Z}^2}\frac{1}{(1+|k|^2)^{1-\lambda}} {\bf e}_k (x) \overline{{\bf e}_k(y)}\right) 
\left( \sum _{k\in {\mathbb Z}^2}\frac{1}{1+|k|^2} {\bf e}_k (x) \overline{{\bf e}_k(y)}\right) ^{n-1} \\ 
&\quad \hspace{10cm} \times \frac{1}{(1+|\ell^2|)^\beta} \overline{{\bf e}_l(x)} {\bf e}_l(y) dx dy \\
&= (2\pi )^n \int _{\Lambda} \int _{\Lambda} K^{1-\lambda} (x,y) (K^1(x,y))^{n-1} K^\beta (x,y) dx dy .
\end{align*}
By using the fact that
\begin{align*}
K^\gamma (x,y) &\leq C_\gamma ( 1+|x-y|^{2\gamma -2}), \quad \gamma \in (0,1),\\
K^1(x,y) &\leq C - \frac{1}{2\pi}\log (1\wedge |x-y|),
\end{align*}
(see \cite[Lemma 5.2]{MR99} or \cite[Proposition A.2]{AKMR19}) and an elementary inequality
\[
(x+c)^n \leq (1+\lambda )^{n-1} x^n + c^n \left( 1+\frac{1}{\lambda}\right) ^{n-1} , \quad x,c \in (0,\infty) ,\ n\in {\mathbb N} , 
\]
we have
\begin{align*}
&\sum _{\ell\in {\mathbb Z}^2} \frac{1}{(1+|\ell|^2)^\beta}
\sum _{\substack{k_1,k_2,\dots, k_n \in {\mathbb Z}^2;\\ k_1+k_2+\dots +k_n =\ell}}
\frac1{(1+|k_1|^2)^{1-\lambda}}\prod _{m=2}^n \frac{1}{1+|k_m|^2}\\
&\leq C(2\pi )^n \int _{|x|<1} \left( C - \frac{1}{2\pi}\log |x| \right)^{n-1} ( 1+|x|^{-2\lambda}) ( 1+|x|^{2\beta -2}) dx +C\\
&\leq C(2\pi)^n \left[ \left(1+\frac1\lambda\right)^{n-1}C^n
+ (1+\lambda )^{n-1} \frac1{(2\pi)^{n-1}}\int _{|x|<1} |x|^{2\beta -2-2\lambda} \left( - \log |x| \right)^{n-1} dx\right] \\
&\leq C^n+ C(1+\lambda)^{n-1}\int _0^1 r^{2(\beta -\lambda)-1} \left( - \log r \right)^{n-1} dr \\
&= C^n + C(1+\lambda)^{n-1}\int _0^\infty t^{n-1} e^{-2(\beta -\lambda) t} dt\\
&\leq C^n+ C\left( \frac{1+\lambda}{2(\beta -\lambda)} \right) ^{n-1} (n-1)!.
\end{align*}
Therefore, if $\alpha ^2 / (4\pi \beta )<1$, by choosing $\lambda \in (0,1)$ sufficiently small we obtain \eqref{ineq:sffcondwexp}.

We show the uniqueness. Let $\{\exp_N^{\diamond,1}(\alpha\phi)\}$ and $\{\exp_N^{\diamond,2}(\alpha\phi)\}$ be the sequences defined by the Fourier multipliers $\psi_1$ and $\psi_2$, respectively. Similarly to calculations above, by using the inequality
$$
|\psi_1(2^{-N}k)-\psi_2(2^{-N}k)|
\le|\psi_1(2^{-N}k)-1|+|\psi_2(2^{-N}k)-1|
\le C2^{-\lambda N}|k|^\lambda,
$$
we can conclude that
\begin{align*}
\int_{\mathcal{D}'(\Lambda)}\left\|\exp_N^{\diamond,1}(\alpha\phi)-\exp_N^{\diamond,2}(\alpha\phi)\right\|_{H^{-\beta}}^2\mu_0(d\phi)
\lesssim2^{-\lambda N}\xrightarrow{N\to\infty}0.
\end{align*}
Hence the limits $\exp^{\diamond,1}(\alpha\phi)$ and $\exp^{\diamond,2}(\alpha\phi)$ coincide as an element of $L^2(\mu_0;H^{-\beta})$.
\end{proof}

\subsection{$\exp(\Phi)_2$-quantum field}

Since $\exp^\diamond(\alpha\cdot)$ is a nonnegative distribution defined $\mu_0$-almost everywhere, we can define the $\exp(\Phi)_2$-measure.

\begin{cor}\label{cor:expmeas}
The $\exp(\Phi)_2$-measure $\mu^{(\alpha)}$ 
defined by \eqref{expphimeas} is well-defined as the limit of the approximating measures $\{\mu_N^{(\alpha)}\}$ defined by \eqref{expNmeas} in weak topology, and 
absolutely continuous with respect to $\mu_0$.
In particular, the support of $\mu^{(\alpha)}$ is in $H^{-\varepsilon}$ for $\varepsilon >0$.
Moreover, the Radon-Nikodym derivatives $\left\{\frac{d\mu_N^{(\alpha)}}{d\mu_0}\right\}$ are uniformly bounded.
\end{cor}

\begin{proof}
From the positivity of $\exp_N^\diamond$, the function $\phi\mapsto\exp\{-\int_\Lambda\exp _N^\diamond(\alpha\phi)(x)dx\}$ is bounded by $1$, $\mu_0$-almost everywhere.
For the normalizing constant, by the dominated convergence theorem and Jensen's inequality,
\begin{align*}
Z^{(\alpha)}
&=\lim_{N\to\infty}\int_{\mathcal{D}'(\Lambda)}\exp \left\{ - \int_\Lambda\exp_N^\diamond(\alpha \phi )(x)dx\right\}\mu_0(d\phi)\\
&\ge\lim_{N\to\infty}\exp \left\{ - \int_{\mathcal{D}'(\Lambda)}
\mu_{0}(d\phi)
\int_\Lambda\exp_N^\diamond(\alpha \phi )(x)dx
\right\} =\exp\left\{-\int_\Lambda dx\right\}
=e^{-(2\pi)^2}>0.
\end{align*}
Here we use the fact that 
$\int_{
\mathcal{D}'(\Lambda)}
\exp_N^\diamond(\alpha\phi)(x)\mu_0(d\phi)=1$ ($x\in \Lambda$), which follows from the definition.
Hence, by the dominated convergence theorem again, $\mu^{(\alpha)}$ is defined as the limit of $\{\mu_N^{(\alpha)}\}$ in weak topology.
Absolute continuity and the boundedness of the Radon-Nikodym derivatives follows from the uniform boundedness of $\phi\mapsto\exp\{-\int_\Lambda\exp _N^\diamond(\alpha\phi)(x)dx\}$.
Absolute continuity of $\mu^{(\alpha)}$ with respect to $\mu_0$ and the fact that the support of $\mu _0$ is in $H^{-\varepsilon }$ for $\varepsilon >0$ immediately imply that the support of $\mu^{(\alpha)}$ is in $H^{-\varepsilon }$ for $\varepsilon >0$.
\end{proof}

\subsection{Wick exponential of the OU process}

For the OU process $X=X(\phi)$, we also define the approximating Wick exponential
\begin{equation*}
{\mathcal X}_t^{(\exp ,N)}(\phi) = \exp_N^\diamond(\alpha X_t(\phi))(x).
\end{equation*}
$\mathcal{X}_t^{(\exp,N)}$ can be regarded as a random variable on the product space $(\Omega\times\mathcal{D}'(\Lambda),\mathbb{P}\otimes \mu_0)$.

\begin{thm}\label{thm:expOU}
Let $|\alpha|<\sqrt{4\pi}$ and $\beta\in(\frac{\alpha^2}{4\pi},1)$.
Then $\{ {\mathcal X}^{(\exp ,N)} \}$ converges in $L^2([0,T];H^{-\beta})$ for any $T>0$, $\mathbb{P}\otimes\mu_0$-almost surely and in $L^2(\mathbb{P}\otimes\mu_0)$.
Moreover, the limit $\mathcal{X}^{(\exp,\infty)}$ is independent of the choice of $\psi$.
\end{thm}

\begin{proof}
The proof is almost the same as that of Theorem \ref{thm:expgff}. By the invariance of $\mu_0$,
\begin{align*}
&\mathbb{E} \bigg [
 \int_{\mathcal{D}'(\Lambda)}\mu_0(d\phi)
\sum _{N=1}^\infty 
\left \{ \int_0^T\left\| \mathcal{X}_t^{(\exp,N+1)}(\phi) - \mathcal{X}_t^{(\exp,N)}(\phi) \right\| _{H^{-\beta}}^2dt
\right \}^{1/2}
\bigg ]
\\
&\le\sum_{N=1}^\infty2^{-\lambda N/2}
+\sum_{N=1}^\infty2^{\lambda N/2}
\mathbb{E}\bigg [ \int_{\mathcal{D}'(\Lambda)}\mu_0(d\phi)
\int_0^T\left\| \mathcal{X}_t^{(\exp,N+1)}(\phi) - \mathcal{X}_t^{(\exp,N)}(\phi) \right\| _{H^{-\beta}}^2dt
\bigg ]
\\
&\le\sum_{N=1}^\infty2^{-\lambda N/2}
+\sum_{N=1}^\infty2^{\lambda N/2}T
\int_{\mathcal{D}'(\Lambda)}\left\| \exp_{N+1}^\diamond(\alpha\phi) - \exp_N^\diamond(\alpha\phi) \right\| _{H^{-\beta}}^2\mu_0(d\phi)
<\infty.
\end{align*}
\end{proof}

We show the ``stability" of $\mathcal{X}^{(\exp,\infty)}$ with respect to $\phi$ in the following sense.

\begin{lem}\label{lem:contiexpOU}
Let $\xi_N$ and $\xi_\infty$ be $H^{-2}$-valued random variables independent to $W$.
Assume that the laws $\nu_N$ and $\nu_\infty$ of $\xi_N$ and $\xi$ respectively are absolutely continuous with respect to $\mu_0$, and their Radon-Nikodym derivatives $\frac{d\nu_N}{d\mu_0}$ and $\frac{d\nu_\infty}{d\mu_0}$ are uniformly bounded over $N$.
If $\xi_N$ converges to $\xi_\infty$ in $H^{-2}$ almost surely, then we have
$$
\mathcal{X}^{(\exp,\infty)}(\xi_N)\to
\mathcal{X}^{(\exp,\infty)}(\xi_\infty)
$$
in $L^2([0,T];H^{-\beta})$ for any $T>0$, in probability.
\end{lem}

\begin{proof}
Let $M\in\mathbb{N}$.
By the mild form \eqref{eq:OUmild} of $X$, we have
\begin{align*}
\|P_MX_t(\xi_N)-P_MX_t(\xi_\infty)\|_{C([0,T];C(\Lambda))}
&\lesssim\sup_{t\in[0,T]}\|e^{\frac12(\triangle-1)t}P_M(\xi_N-\xi_\infty)\|_{H^2}\\
&\le 2^{4M}\|\xi_N-\xi_\infty\|_{H^{-2}}
\xrightarrow{N\to\infty}0,
\end{align*}
almost surely. Hence for any fixed $M\in\mathbb{N}$,
\begin{align*}
\mathcal{X}^{(\exp,M)}(\xi_N)
&=\exp\left(\alpha P_MX(\xi_N)-\frac{\alpha^2}2C_M\right)\\
&\to\exp\left(\alpha P_MX(\xi_\infty)-\frac{\alpha^2}2C_M\right)
=\mathcal{X}^{(\exp,M)}(\xi_\infty)
\end{align*}
in $C([0,T];C(\Lambda))$ almost surely.
On the other hand, since the Radon-Nikodym derivatives 
$\frac{d\nu_N}{d\mu_0}$ and $\frac{d\nu_\infty}{d\mu_0}$ 
are uniformly bounded, by using invariance of $\mu _0$ with respect to $X_t$ we have
\begin{align*}
&\sup_{N\in\mathbb{N}\cup\{\infty\}}\mathbb{E}
\Big [
\|\mathcal{X}^{(\exp,M)}(\xi_N)-\mathcal{X}^{(\exp,\infty)}(\xi_N)\|_{L^2([0,T];H^{-\beta})}^2
\Big ]
\\
&\lesssim \mathbb{E} \left[ \int_{\mathcal{D}'(\Lambda)}\|\mathcal{X}^{(\exp,M)}(\phi)-\mathcal{X}^{(\exp,\infty)}(\phi)\|_{L^2([0,T];H^{-\beta})}^2 \mu_0(d\phi) \right]
\\
&= T \int_{\mathcal{D}'(\Lambda)}\| \exp_M^\diamond(\alpha \phi ) - \exp ^\diamond(\alpha \phi )\|_{H^{-\beta}} ^2 \mu_0(d\phi) .
\end{align*}
Hence, by Corollary \ref{cor:expmeas} we have
\[
\sup_{N\in\mathbb{N}\cup\{\infty\}}\mathbb{E}
\Big [
\|\mathcal{X}^{(\exp,M)}(\xi_N)-\mathcal{X}^{(\exp,\infty)}(\xi_N)\|_{L^2([0,T];H^{-\beta})}
\Big ]
\xrightarrow{M\to\infty}0.
\]
By using the inequality $(a+b)\wedge1\le a+(b\wedge1)$ for $a,b\ge0$, we have
\begin{align*}
&\mathbb{E}\left[\|\mathcal{X}^{(\exp,\infty)}(\xi_N)-
\mathcal{X}^{(\exp,\infty)}(\xi_\infty)\|_{L^2([0,T];H^{-\beta})}\wedge1\right]\\
&\le 2\sup_{N\in\mathbb{N}\cup\{\infty\}}\mathbb{E}
\Big [
\|\mathcal{X}^{(\exp,M)}(\xi_N)-\mathcal{X}^{(\exp,\infty)}(\xi_N)\|_{L^2([0,T];H^{-\beta})}
\Big ]
\\
&\quad+\mathbb{E}\left[\|\mathcal{X}^{(\exp,M)}(\xi_N)-
\mathcal{X}^{(\exp,M)}(\xi_\infty)\|_{L^2([0,T];H^{-\beta})}\wedge1\right].
\end{align*}
In the right-hand side, by letting $N\to\infty$ first and then $M\to\infty$, we have the required convergence result.
\end{proof}


\section{Global well-posendess of the strong solution}\label{sec:wellposed}

In this section, we consider the approximating equation \eqref{expsqe1}.
To show Theorem \ref{mainthm1}, we use the Da Prato-Debussche argument. Precisely, we decompose $\Phi^N=X^N+Y^N$, where $X^N$ and $Y^N$ solve
\begin{align}
\label{eq:X}
&\left\{
\begin{aligned}
\partial _t X_t^N &= \frac 12 (\triangle-1) X_t^N + P_N\dot W_t ,\\
X_0^N&=P_N\phi,
\end{aligned}
\right.\\[5pt]
\label{eq:Y}
&\left\{
\begin{aligned}
\partial _t Y_t^N &= \frac 12 (\triangle-1) Y_t^N
- \frac\alpha2 \exp(\alpha Y_t^N)\exp\left(\alpha X_t^N -\frac{\alpha^2}2C_N\right) ,\\
Y_0^N&=0.
\end{aligned}
\right.
\end{align}
Note that $X^N=P_NX(\phi)$, where $X(\phi)$ is the solution of \eqref{eq:OU} with the initial value $\phi$. Hence the renormalized exponential of $X^N$ in the latter equation is equal to
$$
\exp\left(\alpha X_t^N -\frac{\alpha^2}2C_N\right)=\mathcal{X}_t^{(\exp,N)}(\phi).
$$
Since $\mathcal{X}^{(\exp,N)}$ converges to an $L^2([0,T];H^{-\beta})$-valued nonnegative random variable $\mathcal{X}^{(\exp,\infty)}$, in this section we consider the \emph{deterministic} equation
\begin{align*}
\partial_t\Upsilon_t&=\frac12(\triangle-1)\Upsilon_t-\frac\alpha2 e^{\alpha \Upsilon_t}\mathcal{X}_t
\end{align*}
for any generic nonnegative $\mathcal{X}\in L^2([0,T];H^{-\beta})$.

\subsection{Products of continuous functions and nonnegative distributions}\label{sec:nnegdist}

A distribution $\xi\in\mathcal{D}'(\Lambda)$ is said to be \emph{nonnegative} if $\xi(\varphi)\ge0$ for any nonnegative $\varphi\in\mathcal{D}(\Lambda)$.
The product of $f\in C(\Lambda)$ and $\xi\in\mathcal{D}'(\Lambda)$ is ill-defined in general, but if $\xi$ is nonnegative, then such product is well-defined in the following sense.

\begin{thm}[{\cite[Theorem 6.22]{LL01}}]\label{thm:LL}
For any nonnegative $\xi\in\mathcal{D}'(\Lambda)$, there exists a unique nonnegative Borel measure $\mu_\xi$ such that
$$
\xi(\varphi)=\int_{\Lambda}\varphi(x)\mu_\xi(dx),\quad
\varphi\in\mathcal{D}(\Lambda).
$$
Consequently, the domain of $\xi$ is extended to $C(\Lambda)$.
\end{thm}

\begin{defi}
For any nonnegative $\xi\in\mathcal{D}'(\Lambda)$ and any $f\in C(\Lambda)$, we define the Borel measure
$$
\mathcal{M}(f,\xi)(dx):=f(x)\mu_\xi(dx)
$$
where $\mu _\xi (dx)$ is the measure obtained in Theorem \ref{thm:LL}.
\end{defi}

We prove some properties of $\mathcal{M}$. First we recall the following basic result.

\begin{prop}[{\cite[Theorem 2.34]{BCD11}}]\label{prop:heat-besov}
For any $s>0$ and $p,q\in[1,\infty]$, one has the equivalence of norms
$$
\|\xi\|_{B_{p,q}^{-s}}\simeq\|e^{\triangle}\xi\|_{L^p(\Lambda)}+\left\|t^{\frac{s}2}\|e^{t\triangle}\xi\|_{L^p(\Lambda)}\right\|_{L^q([0,1];\frac{dt}t)}.
$$
\end{prop}

\begin{thm}[\cite{Gar18}]\label{thm:func*dist}
Let $s>0$ and $p,q\in[1,\infty]$. There exists a constant $C>0$ such that, one has
$$
\|\mathcal{M}(f,\xi)\|_{B_{p,q}^{-s}}\le C\|f\|_{C(\Lambda)}\|\xi\|_{B_{p,q}^{-s}}
$$
for any nonnegative $\xi\in B_{p,q}^{-s}$ and $f\in C(\Lambda)$.
\end{thm}

\begin{proof}
Since the heat kernel $p_t(x,y)$ associated with $e^{t\triangle}$ is positive, we have
\begin{align*}
|e^{t\Delta}\mathcal{M}(f,\xi)(x)|
&=\left|\int_{\Lambda} p_t(x,y)f(y)\mu_\xi(dy)\right|\\
&\le\|f\|_{C(\Lambda)}\int_{\Lambda}
p_t(x,y)\mu_\xi(dy)
=\|f\|_{C(\Lambda)}(e^{t\Delta}\xi)(x).
\end{align*}
Hence the result follows from Proposition \ref{prop:heat-besov}.
\end{proof}

\begin{thm}\label{thm:stableM}
Let $s>0$ and $p,q\in[1,\infty]$. Denote by $B_{p,q}^{-s,+}$ the subspace of all nonnegative elements in $B_{p,q}^{-s}$. The map
$$
\mathcal{M}:C(\Lambda)\times B_{p,q}^{-s,+}\to B_{p,q}^{-s}
$$
is continuous.
\end{thm}

\begin{proof}
The continuity with respect to $f\in C(\Lambda)$ is obvious from Theorem \ref{thm:func*dist}.
Here we show the continuity with respect to $\xi\in B_{p,q}^{-s,+}$. Fix $f\in C(\Lambda)$ and let $\{\xi_N\}_{N\in\mathbb{N}}$ be an arbitrary sequence in $B_{p,q}^{-s,+}$ such that $\xi_N\to\xi$ in $B_{p,q}^{-s}$.
Since $C^\infty(\Lambda)$ is dense in $C(\Lambda)$, for any $\varepsilon>0$, there exists $g\in C^\infty(\Lambda)$ such that $\|f-g\|_{C(\Lambda)}<\varepsilon$.
As stated in \cite[Theorems 2.82 and 2.85]{BCD11}, the product map
$$
B_{\infty,\infty}^{s+1}\times B_{p,q}^{-s}\ni(g,\xi)\mapsto g\xi\in B_{p,q}^{-s}
$$
is continuous and coincides with $\mathcal{M}(g,\xi)$ if $(g,\xi)\in C^\infty(\Lambda)\times B_{p,q}^{-s,+}$, so we have
\begin{align*}
&\|\mathcal{M}(f,\xi_N)-\mathcal{M}(f,\xi)\|_{B_{p,q}^{-s}}\\
&\le\|\mathcal{M}(f-g,\xi_N)\|_{B_{p,q}^{-s}}+\|\mathcal{M}(f-g,\xi)\|_{B_{p,q}^{-s}}
+\|g(\xi^N-\xi)\|_{B_{p,q}^{-s}}\\
&\lesssim\varepsilon\|\xi_N\|_{B_{p,q}^{-s}}
+\varepsilon\|\xi\|_{B_{p,q}^{-s}}
+\|g\|_{B_{\infty,\infty}^{s+1}}\|\xi_N-\xi\|_{B_{p,q}^{-s}}.
\end{align*}
Letting $N\to\infty$,
\begin{align*}
\limsup_{N\to\infty}\|\mathcal{M}(f,\xi_N)-\mathcal{M}(f,\xi)\|_{B_{p,q}^{-s}}
\lesssim\varepsilon\|\xi\|_{B_{p,q}^{-s}}.
\end{align*}
Since $\varepsilon$ is arbitrary, we have
\begin{align*}
\lim_{N\to\infty}\|\mathcal{M}(f,\xi_N)-\mathcal{M}(f,\xi)\|_{B_{p,q}^{-s}}=0.
\end{align*}
Thus we have the continuity with respect to $\xi\in B_{p,q}^{-s,+}$
\end{proof}

%
%

The time-dependent version of Theorem \ref{thm:stableM} has an important role in this paper.

\begin{thm}\label{replaced keythm}
Let $s>0$, $p,q\in[1,\infty]$, and $r\in(1,\infty]$.
For any time-dependent $(Y,\mathcal{X}) \in L^1([0,T];C(\Lambda))\times L^r([0,T];B_{p,q}^{-s,+})$ and any function $f\in C_b^1(\mathbb{R})$, consider the time-dependent distribution
$$
\mathcal{M}(f(Y),\mathcal{X})(t)
:=\mathcal{M}(f(Y_t),\mathcal{X}_t).
$$
Then the correspondence $(Y,\mathcal{X})\mapsto\mathcal{M}(f(Y),\mathcal{X})$
is well-defined as a map
$$
L^1([0,T];C(\Lambda))\times L^r([0,T];B_{p,q}^{-s,+}) \to L^r([0,T];B_{p,q}^{-s}).
$$
Moreover, it is continuous as a map
$$
L^1([0,T];C(\Lambda))\times L^r([0,T];B_{p,q}^{-s,+}) \to L^{r'}([0,T];B_{p,q}^{-s})
$$
for any $r'\in[1,r)$.
\end{thm}

\begin{proof}
Since $f(Y)\in L^\infty([0,T];C(\Lambda))$, by Theorem \ref{thm:func*dist} the product $\mathcal{M}(f(Y),\mathcal{X})$ is well-defined and
\begin{align}\label{replaced keythm uniform integrable}
\|\mathcal{M}(f(Y),\mathcal{X})\|_{L^r([0,T];B_{p,q}^{-s})}
\lesssim\|f\|_{\infty}\|\mathcal{X}\|_{L^r([0,T];B_{p,q}^{-s})}.
\end{align}
Next we show the convergence. Let $\{(Y^N,\mathcal{X}^N)\}_{N\in\mathbb{N}}$ be an arbitrary sequence such that
\begin{itemize}
\item $\{Y^N\}_N$ is a sequence of measurable functions on $[0,T]\times\Lambda$ such that
\begin{align*}
Y^N\to Y
\quad\text{in $L^1([0,T];C(\Lambda))$},
\end{align*}
\item $\{\mathcal{X}^N\}_N$ is a sequence of $L^r([0,T];B_{p,q}^{-s,+})$ such that
\begin{align*}
\mathcal{X}^N\to \mathcal{X}
\quad\text{in $L^r([0,T];B_{p,q}^{-s})$}.
\end{align*}
\end{itemize}
Since $f'$ is bounded,
\begin{align*}
\|f(Y^N)-f(Y)\|_{L^1([0,T];C(\Lambda))}
\le\|f'\|_\infty\|Y^N-Y\|_{L^1([0,T];C(\Lambda))}\xrightarrow{N\to\infty}0.
\end{align*}
In particular,
\begin{align*}
\left\{
\begin{aligned}
f(Y_t^N)&\to f(Y_t)
\quad\text{in $C(\Lambda)$},\\
\mathcal{X}_t^N&\to\mathcal{X}_t
\quad\text{in $B_{p,q}^{-s,+}$},
\end{aligned}
\right.
\end{align*}
for almost every $t\in[0,T]$. Hence by Theorem \ref{thm:stableM} we have
\begin{align*}
\mathcal{M}(f(Y_t^N),\mathcal{X}_t^N)
\to \mathcal{M}(f(Y_t),\mathcal{X}_t)
\quad\text{in $B_{p,q}^{-s}$}
\end{align*}
for almost every $t\in[0,T]$.
Note that $\mathcal{M}(f(Y^N),\mathcal{X}^N)$ is bounded in $L^r([0,T];B_{p,q}^{-s})$ by the estimate \eqref{replaced keythm uniform integrable}.
Since $r>1$, for any $r'\in[1,r)$, the function $\|\mathcal{M}(f(Y_t^N),\mathcal{X}_t^N)\|_{B_{p,q}^{-s}}^{r'}$ is uniformly integrable. By Lebesgue's convergence theorem, we have
\begin{align*}
\int_0^T\|\mathcal{M}(f(Y_t^N),\mathcal{X}_t^N)-\mathcal{M}(f(Y_t),\mathcal{X}_t)\|_{B_{p,q}^{-s}}^{r'} dt
\xrightarrow{N\to\infty}0.
\end{align*}
\end{proof}

\subsection{Global well-posedness of $\Upsilon$}

We fix the parameters $\beta\in(0,1)$ and $T>0$.
In this section, we consider the initial value problem
\begin{align}
\label{eq:Ups}
\left\{
\begin{aligned}
\partial_t\Upsilon_t&=\frac12(\triangle-1)\Upsilon_t-\frac\alpha2 \mathcal{M}(e^{\alpha \Upsilon_t},\mathcal{X}_t),\\
\Upsilon_0&=\upsilon,
\end{aligned}
\right.
\end{align}
for any given $\mathcal{X}\in L^2([0,T];H_+^{-\beta})$ and $\upsilon\in H^{2-\beta}$.
We denote by $H_+^{-\beta}$ the subspace of all nonnegative elements in $H^{-\beta}$.
To solve the equation \eqref{eq:Ups}, we introduce the space
\begin{align*}
\mathscr{Y}_T&=\left\{\Upsilon\in L^2([0,T];C(\Lambda)\cap H^1)\cap C([0,T];L^2(\Lambda))\ ;
e^{\alpha\Upsilon}\in L^\infty([0,T];C(\Lambda))
\right\}.
\end{align*}
Our aim is to show the following theorem.

\begin{thm}\label{thm:solmapUps}
Let $\mathcal{X}\in L^2([0,T];H_+^{-\beta})$ and $\upsilon\in H^{2-\beta}$.
Then there exists a unique mild solution $\Upsilon\in\mathscr{Y}_T$ of \eqref{eq:Ups}, that is, the equation
\begin{align}\label{eq:mildUps}
\Upsilon_t=e^{\frac12(\triangle-1)t}\upsilon-\frac\alpha2\int_0^te^{\frac12(\triangle-1)(t-s)}\mathcal{M}(e^{\alpha\Upsilon_s},\mathcal{X}_s)ds
\end{align}
holds for any $t\in(0,T]$.
Moreover, this solution belongs to the space $L^2([0,T];H^{1+\delta})\cap C([0,T];H^\delta)$ for any $\delta\in(0,1-\beta)$, and the mapping
$$
\mathcal{S}:H^{2-\beta}\times L^2([0,T];H_+^{-\beta})\ni(\upsilon,\mathcal{X})
\mapsto \Upsilon\in L^2([0,T];H^{1+\delta})\cap C([0,T];H^\delta)
$$
is continuous.
\end{thm}

We first show the uniqueness of the solution.

\begin{lem}\label{lem:unique}
For any $\mathcal{X}\in L^2([0,T];H_+^{-\beta})$ and $\upsilon\in H^{2-\beta}$, there is at most one mild solution $\Upsilon\in\mathscr{Y}_T$ of the equation \eqref{eq:Ups}.
\end{lem}

\begin{proof}
Let $\Upsilon,\Upsilon'\in\mathscr{Y}_T$ be two solutions of \eqref{eq:Ups} with the same $\mathcal{X}$ and $\upsilon$. Then $Z=\Upsilon-\Upsilon'$ solves the equation
\begin{align*}
\left\{\partial_t-\frac12(\triangle-1)\right\}Z_t
&=-\frac\alpha2\mathcal{M}(e^{\alpha \Upsilon_t}-e^{\alpha \Upsilon_t'},\mathcal{X}_t)
=:D_t.
\end{align*}
Since $e^{\alpha\Upsilon},e^{\alpha\Upsilon'}\in L^\infty([0,T];C(\Lambda))$ and $\mathcal{X}\in L^2([0,T];H_+^{-\beta})$, we have that $D\in L^2([0,T];H^{-\beta})$ by Theorem \ref{thm:func*dist}.
Let $\lambda >0$ and define $Z^\lambda =e^{\lambda \triangle}Z$. Then $Z^\lambda$ solves the equation
\begin{align*}
\left\{\partial_t-\frac12(\triangle-1)\right\}Z_t^\lambda
&=e^{\lambda \triangle}D.
\end{align*}
By the regularizing effect (see Proposition \ref{prop:heatsemigr}), $e^{\lambda \triangle}D$ belongs to $L^2([0,T];C^\infty(\Lambda))$. Then by the Schauder estimate (see Proposition \ref{prop:Schauder}), we have that $Z^\lambda$ belongs to $C^{1-\kappa}([0,T];C^\infty(\Lambda))$ for any $\kappa>0$.
Hence we can justify the energy equation
\begin{align*}
&\int_\Lambda|Z_t^\lambda (x)|^2dx
=2\int_0^t\int_\Lambda Z_s^\lambda (x)\partial_sZ_s^\lambda (x)dxds\\
&=-\int_0^t\int_\Lambda|\nabla Z_s^\lambda (x)|^2dxds-\int_0^t\int_\Lambda|Z_s^\lambda (x)|^2dxds
+2\int_0^t\int_\Lambda Z_s^\lambda (x)e^{\lambda \triangle}D_s(x)dxds
\end{align*}
where the first equality is justified as a Young's integral. Letting $\lambda\to0$, we have
\begin{align*}
&\int_\Lambda|Z_t(x)|^2dx\\
&=-\int_0^t\int_\Lambda|\nabla Z_s(x)|^2dxds-\int_0^t\int_\Lambda|Z_s(x)|^2dxds
+2\int_0^t\int_\Lambda Z_s(x)D_s(x)dxds.
\end{align*}
For the last term,
\begin{align*}
2\int_\Lambda Z_s(x)D_s(x)dx
&=-\alpha\int_\Lambda(e^{\alpha\Upsilon_s(x)}-e^{\alpha\Upsilon_s'(x)})Z_s(x) \mu _{\mathcal{X}_s} (dx)\\
&=-\alpha^2\int_\Lambda e^{A(\alpha\Upsilon_s(x),\alpha\Upsilon_s(x))}|Z_s(x)|^2 \mu _{\mathcal{X}_s} (dx)\le0,
\end{align*}
where $\mu _{\mathcal{X}_s}$ is the measure appeared in Theorem \ref{thm:LL} and $A(x,y)$ is a continuous function on $\mathbb{R}^2$ defined by
$$
A(x,y)=
\begin{cases}
\log\frac{e^x-e^y}{x-y},&x\neq y,\\
x&x=y.
\end{cases}
$$
Hence we have $\|Z_t\|_{L^2(\Lambda)}=0$ for any $t\in(0,T]$, which implies $\Upsilon=\Upsilon'$ in $\mathscr{Y}_T$.
\end{proof}

Next we show the existence. The following embedding is frequently used below.

\begin{lem}[{\cite[Corollary~5]{JSimon}}]\label{Simon}
Let $A\subset B\subset C$ be Banach spaces such that the inclusion $A\hookrightarrow B$ is compact. Let $p,r\in[1,\infty]$ and $s> \max\{ 0, \frac1r-\frac1p\}$.
Then the embedding
$$
L^p([0,T];A)\cap W^{s,r}([0,T];C) \hookrightarrow L^p([0,T];B)
$$
is compact. When $p=\infty$ (resp. $r=\infty$), the norm $L^p([0,T];\cdot)$ (resp. $W^{s,r}([0,T];\cdot)$) is replaced by $C([0,T];\cdot)$ (resp. $C^s([0,T];\cdot)$). 
\end{lem}

\begin{lem}\label{lem:existence}
For any $\mathcal{X}\in L^2([0,T];H_+^{-\beta})$ and $\upsilon\in H^{2-\beta}$, there is at least one mild solution $\Upsilon\in\mathscr{Y}_T$. Moreover, for any $\delta\in(0,1-\beta)$, there exists a constant $C>0$ independent of $\mathcal{X}$ and $\upsilon$ such that one has the a priori estimate
\begin{align}\label{eq:aprioriUps}
\begin{aligned}
&\|\Upsilon\|_{L^2([0,T];H^{1+\delta})\cap C([0,T];H^\delta)\cap C^{\delta/2}([0,T];L^2)}\\
&\qquad\le C\left\{\|\upsilon\|_{H^{2-\beta}}+e^{|\alpha|\|\upsilon\|_{C(\Lambda)}}\|\mathcal{X}\|_{L^2([0,T];H^{-\beta})}\right\}.
\end{aligned}
\end{align}
\end{lem}

\begin{proof}
Let $\{\mathcal{X}^N\}_{N\in\mathbb{N}}$ be a family of nonnegative continuous functions on $[0,T]\times\Lambda$ converging to $\mathcal{X}$ in $L^2([0,T];H^{-\beta})$. Such an approximation exists. Indeed, if $\eta$ is a nonnegative continuous function on $\mathbb{R}$ supported in $[-1,1]$ and such that $\int _{-\infty}^\infty \eta (s) ds=1$, then the nonnegative continuous function
$$
\mathcal{X}_t^N(x):=N\int_0^T\eta(N(t-s))(e^{\frac1N\triangle}\mathcal{X}_s)(x)ds
$$
converges to $\mathcal{X}$ in $L^2([0,T];H^{-\beta})$ as $N\to\infty$.
Now we consider the classical global solutions of the approximating equations
\begin{align}
\label{eq:UpsN}
\left\{
\begin{aligned}
\partial_t\Upsilon_t^N&=\frac12(\triangle-1)\Upsilon_t^N-\frac\alpha2 e^{\alpha \Upsilon_t^N}\mathcal{X}_t^N,\\
\Upsilon_0^N&=\upsilon.
\end{aligned}
\right.
\end{align}
Note that $\upsilon\in H^{2-\beta}\subset C(\Lambda)$ by the Sobolev embedding.
By using the mild form, if $\alpha>0$, we have
\begin{align*}
\Upsilon_t^N&=e^{\frac12(\triangle-1)t}\upsilon-\frac\alpha2\int_0^te^{\frac12(\triangle-1)(t-s)}e^{\alpha \Upsilon_s^N}\mathcal{X}_s^Nds\\
&\le e^{\frac12(\triangle-1)t}\upsilon\le\|\upsilon\|_{C(\Lambda)},
\end{align*}
and if $\alpha<0$, we have
\begin{align*}
\Upsilon_t^N
&\ge e^{\frac12(\triangle-1)t}\upsilon\ge-\|\upsilon\|_{C(\Lambda)}.
\end{align*}
These yield 
\begin{align}\label{eq:bdUpsN}
\|e^{\alpha \Upsilon^N}\|_{C([0,T];C(\Lambda))}\le e^{|\alpha|\|\upsilon\|_{C(\Lambda)}}.
\end{align}

Let $\delta<\delta'<1-\beta$. Applying the Schauder estimate (Proposition \ref{prop:Schauder}) to $\Upsilon^N$,
\begin{align*}
&\|\Upsilon^N\|_{L^2([0,T];H^{1+\delta'})\cap C([0,T];H^{\delta'})\cap C^{\delta'/2}([0,T];L^2)}\\
&\lesssim \left(\|\upsilon\|_{H^{2-\beta}}+\|\mathcal{M}(e^{\alpha \Upsilon^N},\mathcal{X}^N)\|_{L^2([0,T];H^{-\beta})}\right)\\
&\lesssim \left(\|\upsilon\|_{H^{2-\beta}}+\|e^{\alpha \Upsilon^N}\|_{L^\infty(0,T;C(\Lambda))}\|\mathcal{X}^N\|_{L^2([0,T];H^{-\beta})}\right)\\
&\lesssim \left(\|\upsilon\|_{H^{2-\beta}}+e^{|\alpha|\|\upsilon\|_{C(\Lambda)}}\|\mathcal{X}^N\|_{L^2([0,T];H^{-\beta})}\right).
\end{align*}
By Lemma \ref{Simon}, the embeddings
\begin{align*}
L^2([0,T];H^{1+\delta'})\cap C^{\delta'/2}([0,T];L^2)&\hookrightarrow
L^2([0,T];H^{1+\delta}),\\
C([0,T];H^{\delta'})\cap C^{\delta'/2}([0,T];L^2)&\hookrightarrow
C([0,T];H^{\delta})
\end{align*}
are compact. Hence there exists a subsequence $\{N_k\}$ such that
\begin{align}\label{conv:Upsexists}
\Upsilon^{N_k}\to \Upsilon\quad\text{in $L^2([0,T];H^{1+\delta})\cap C([0,T];H^{\delta})$}.
\end{align}
In particular, we have \eqref{eq:aprioriUps} for $\Upsilon$ and $\|e^{\alpha\Upsilon}\|_{L^\infty([0,T];C(\Lambda))}\le e^{|\alpha|\|\upsilon\|_{C(\Lambda)}}$ by \eqref{eq:bdUpsN} and Fatou's lemma.

We show that $\Upsilon$ solves the mild equation \eqref{eq:mildUps}.
Since $(\Upsilon^{N_k},\mathcal{X}^{N_k})$ satisfies \eqref{eq:mildUps}, it is sufficient to show
\begin{align}\label{eq:convMUps}
\mathcal{M}(e^{\alpha \Upsilon^{N_k}},\mathcal{X}^{N_k})\to
\mathcal{M}(e^{\alpha \Upsilon},\mathcal{X})
\quad\text{in $L^{2-\kappa}([0,T];H^{-\beta})$}
\end{align}
for some $\kappa>0$.
Then letting $N_k\to\infty$ on both sides of \eqref{eq:mildUps} and applying Proposition \ref{prop:Schauder}, we have the same equality for $(\Upsilon,\mathcal{X})$ in the space $C([0,T];H^\delta)$.
Now \eqref{eq:convMUps} is an immediate consequence of Theorem \ref{replaced keythm}.
Indeed, by \eqref{conv:Upsexists} and the embedding $H^{1+\delta}\subset C(\Lambda)$,
\begin{align*}
\alpha\Upsilon^{N_k}&\to \alpha\Upsilon
\quad\text{in $L^2([0,T];C(\Lambda))$}.
\end{align*}
Moreover, since $\alpha\Upsilon^{N_k}$ is uniformly bounded from above (see (\ref{eq:bdUpsN})), we can apply Theorem \ref{replaced keythm} to a function $f\in C_b^1(\mathbb{R})$ such that $f(x)=e^x$ on a subset $x\in(-\infty,a]$ for some fixed $a\in\mathbb{R}$.

Thus we have the existence of the mild solution. Applying Proposition \ref{prop:Schauder}, the unique solution $\Upsilon$ also satisfies the a priori estimate \eqref{eq:aprioriUps}.

\end{proof}

We obtained that the solution map $\mathcal{S}:(\upsilon,\mathcal{X})\mapsto\Upsilon$ is well-defined. Finally we show the stability of the map $\mathcal{S}$.

\begin{proof}[{\bfseries Proof of Theorem \ref{thm:solmapUps}}]
Fix $(\upsilon,\mathcal{X}) \in H^{2-\beta}\times L^2([0,T];H_+^{-\beta})$ and let $\{(\upsilon^N,\mathcal{X}^N)\}_{N\in\mathbb{N}}\subset H^{2-\beta}\times L^2([0,T];H_+^{-\beta})$ be an arbitrary sequence which converges to $(\upsilon,\mathcal{X})$.
By using the a priori estimate \eqref{eq:aprioriUps} for $\Upsilon^N=\mathcal{S}(\upsilon^N,\mathcal{X}^N)$ and for $\delta''\in(\delta,1-\beta)$, we have
$$
\sup_N\|\Upsilon ^N\|_{L^2([0,T];H^{1+\delta''})\cap C([0,T];H^{\delta''})\cap C^{\delta''/2}([0,T];L^2)}<\infty.
$$
Thus we are in just a similar situation to the proof of Lemma \ref{lem:existence}.
By the compactness argument and by Theorem \ref{replaced keythm}, there exists a subsequence $\{\Upsilon^{N_k}\}$ which converges in the space $L^2([0,T];H^{1+\delta})\cap C([0,T];H^{\delta})$ and its limit coincides with $\Upsilon=\mathcal{S}(\upsilon,\mathcal{X})$.
This yields that any subsequence $\{\Upsilon^{N_k}\}$ has a subsequence $\{\Upsilon^{N_{k_\ell}}\}$ which converges to the common $\Upsilon$.
Thus we can conclude that
$$
\Upsilon^N\to\Upsilon
\quad\text{in $L^2([0,T];H^{1+\delta})\cap C([0,T];H^{\delta})$},
$$
which means the continuity of the map $\mathcal{S}$.
\end{proof}

\subsection{Proof of Theorem \ref{mainthm1}}

Now the first main result immediately follows.

\begin{proof}[{\bfseries Proof of Theorem \ref{mainthm1}}]
By the Da Prato-Debussche decomposition \eqref{eq:X}-\eqref{eq:Y}, the solution $\Phi^N(\phi)$ of the equation \eqref{expsqe1} has the form
$$
\Phi^N(\phi)=P_NX(\phi)+\mathcal{S}(0,\mathcal{X}^{(\exp,N)}(\phi)).
$$
For $\mu_0$-a.e. $\phi$, the first term in the right-hand side converges almost surely to $X(\phi)$ in $C([0,T];H^{-\varepsilon})$ by Proposition \ref{prop:aprioriOU}, and the second term converges almost surely to $\mathcal{S}(0,\mathcal{X}^{(\exp,\infty)}(\phi))$ in $C([0,T];H^\delta)$ by Theorem \ref{thm:expOU} and Theorem \ref{thm:solmapUps}. Hence $\Phi^N(\phi)$ converges to
$$
\Phi(\phi)=X(\phi)+\mathcal{S}(0,\mathcal{X}^{(\exp,\infty)}(\phi))
$$
in the space $C([0,T];H^{-\varepsilon})$ almost surely, for $\mu_0$-a.e. $\phi$.
\end{proof}

\section{Stationary solution}\label{sec:stationary}

In this section, we consider the SPDE \eqref{expsqe2}.
We first note that the generator of ${\bf\Phi}^N$ on ${\mathfrak F}C^{\infty}_{b}$ is given by
\begin{align*}
\mathcal{L}_{N}F(\phi)
&=\frac12\sum_{i, j=1}^n \partial_i \partial_jf(\langle \phi,{l_1}\rangle,\dots,\langle \phi,{l_n}\rangle) 
\langle l_{i}, l_{j} \rangle
\\
&\quad-\frac12\sum_{j=1}^n \partial_jf(\langle \phi,{l_1}\rangle,\dots,\langle \phi,{l_n}\rangle)
\cdot \big \{
{\big \langle} (1-\triangle)\phi, l_{j}
{\big \rangle}
+\alpha
{\big \langle}
P_{N}
\exp^\diamond_{N}(\alpha \phi), {l_j} {\big \rangle}
\big \},
\end{align*}
where $F(\phi)=f(\langle \phi,{l_1}\rangle,\dots,\langle \phi,{l_n}\rangle)$
with $f\in C^{\infty}_{b}(\mathbb R^{n}), l_{1}, \ldots l_{n} \in {\rm Span}\{e_{k};
k\in {\mathbb Z}^{2} \}$.
Applying the integration by parts formula for $\mu_{N}^{(\alpha)}$, we have
\begin{equation}
\int_{\mathcal{D}'(\Lambda)}\mathcal{L}_NF(\phi)G(\phi)\mu_N^{(\alpha)}(d\phi)
=\frac{1}{2} \int_{\mathcal{D}'(\Lambda)}
 (D_{H}F(\phi), D_{H}G(\phi))_{H} \mu_{N}^{(\alpha)}(d\phi)
\nonumber
\end{equation}
for $F, G\in {\mathfrak F}C^{\infty}_{b}$.
Hence by putting $G=1$ and applying Echeverr\'ia's criterion \cite{Ech82},  
we obtain that $\mu_N^{(\alpha)}$ is an invariant measure of the process ${\bf\Phi}^N$.
Therefore, if $\xi_N$ be a random variable with the law $\mu_N^{(\alpha)}$ and independent of $W$, then 
$\bar{\bf\Phi}^N={\bf\Phi}(\xi_N)$ is a stationary process.
In this section, we show the convergence of $\{\bar{\bf\Phi}^N\}$ in law.

\subsection{Tightness of stationary solutions}

We show the tightness of $\{\bar{\bf\Phi}^N\}$.
By the definition \eqref{eq:OU} of the OU process $X$, we can decompose $\bar{\bf\Phi}^N=X(\xi_N)+{\bf Y}^N$, where ${\bf Y}^N$ solves
\begin{align}
\label{eq:statY}
&\left\{
\begin{aligned}
\partial _t {\bf Y}_t^N&= \frac 12 (\triangle-1) {\bf Y}_t^N - \frac\alpha2 P_N
\left\{\exp(\alpha P_N{\bf Y}_t^N)
\exp\left(\alpha P_NX_t(\xi_N)-\frac{\alpha^2}2C_N\right)\right\},\\
{\bf Y}_0^N&=0.
\end{aligned}
\right.
\end{align}
For $X(\xi_N)$, by the a priori estimate of the OU process
 (Proposition \ref{prop:aprioriOU}) and the uniform bound
$$
\sup_{N\in\mathbb{N}}\mathbb{E} \big [ \|\xi_N\|_{H^{-\varepsilon}} \big ]
=\sup_{N\in\mathbb{N}}\int_{\mathcal{D}'(\Lambda)}\|\phi\|_{H^{-\varepsilon}}\mu_N^{(\alpha)}(d\phi)<\infty,
$$
it is easy to check that
\begin{align}\label{ineq:tightX}
\sup_{N\in\mathbb{N}}\mathbb{E}\left[ \left\| X_0(\xi_N) \right\| _{H^{-\varepsilon}} \right]
+\sup_{N\in\mathbb{N}} \mathbb{E}\left[ \sup _{s,t \in [0,T]} \frac{\left\| X_t(\xi_N)  - X_s(\xi_N) \right\| _{H^{-\varepsilon}}}{|t-s|^\lambda} \right] \leq C.
\end{align}
for sufficiently small $\lambda ,\varepsilon>0$.
Next we show the uniform bound of ${\bf Y}^N$.

\begin{prop}\label{prop:tightY}
For $\lambda \in (0, (1-\beta )/2)$, we have
$$
\sup_{N\in\mathbb{N}} \mathbb{E}\left[ \sup _{s,t \in [0,T]} \frac{\left\| {\bf Y}_t^N  - {\bf Y}_s^N \right\| _{L^2}}{|t-s|^\lambda } \right] \leq C.
$$
\end{prop}

\begin{proof}
Note that the renormalized exponential in the right hand side of \eqref{eq:statY} is equal to
$$
\exp\left(\alpha P_NX_t(\xi_N)-\frac{\alpha^2}2C_N\right)
=\mathcal{X}_t^{(\exp,N)}(\xi_N).
$$
Similarly to the proof of Lemma \ref{lem:existence}, we have
\begin{align*}
\|e^{\alpha P_N{\bf Y}^N}\|_{C([0,T];C(\Lambda))}\lesssim 
\exp \big( {|\alpha|\|{\bf Y}_0^N\|_{C(\Lambda)}} \big)=1,
\end{align*}
so by the Schauder estimate (see Proposition \ref{prop:Schauder}) and Theorem \ref{thm:func*dist},
\begin{align*}
\mathbb{E}\big [  \|{\bf Y}^N\|_{C^\lambda([0,T];L^2)} \big ]
&\lesssim
\mathbb{E}\Big [  \left\|P_N\left\{e^{\alpha P_N{\bf Y}^N}
\mathcal{X}^{(\exp,N)}(\xi_N)\right\}\right\|_{L^2([0,T];H^{-\beta})}
\Big ]
\\
&\lesssim
\mathbb{E} \Big [  \|e^{\alpha P_N{\bf Y}^N}\|_{C([0,T];C(\Lambda))}
\left\|\mathcal{X}^{(\exp,N)}(\xi_N)\right\|_{L^2([0,T];H^{-\beta})}
\Big ]
\\
&\lesssim
\mathbb{E}
\Big [
\left\|\mathcal{X}^{(\exp,N)}(\xi_N)\right\|_{L^2([0,T];H^{-\beta})}
\Big ].
\end{align*}
Since the Radon-Nikodym derivative 
$\frac{d\mu_N^{(\alpha)}}{d\mu_0}$ is uniformly bounded (see Corollary \ref{cor:expmeas}), 
\begin{align*}
\sup_N\mathbb{E}\Big[  \left\|\mathcal{X}^{(\exp,N)}(\xi_N)\right\|_{L^2([0,T];H^{-\beta})} \Big ]
\lesssim\sup_N E\left[ \int_{\mathcal{D}'(\Lambda)}
\left\|\mathcal{X}^{(\exp,N)}(\phi)\right\|_{L^2([0,T];H^{-\beta})}\mu_0(d\phi) \right] <\infty.
\end{align*}
Hence we obtain the required estimate.
\end{proof}

\begin{thm}
The laws of ${\bar{\bf\Phi}^N}$ in $C([0,T],H^{-\varepsilon})$ are tight. Moreover, for any subsequence $\{\bar{\bf\Phi}^{N_k}\}$ which converges to a process $\bar{\bf\Phi}$ in law, the law of $\bar{\bf\Phi}_t$ is $\mu^{(\alpha )}$ for any $t\ge0$.
\end{thm}

\begin{proof}
By \eqref{ineq:tightX}, Proposition \ref{prop:tightY}, and Chebyshev's inequality, for $h\in (0,1]$ and $\kappa \in (0,1]$, we have
\begin{align*}
\sup _{N\in {\mathbb N}}
{\mathbb P} \left( \sup _{\substack{s,t\in [0,T];\\|s-t|<h}}
\left\| \bar{\bf\Phi}_t^N - \bar{\bf\Phi}_s^N \right\| _{H^{-\varepsilon}} > \kappa \right)
&\leq \frac{h^\lambda}{\kappa}
\mathbb{E}\left[\sup _{\substack{s,t\in [0,T];\\|s-t|<h}}
\frac{\left\| \bar{\bf\Phi}_t^N - \bar{\bf\Phi}_s^N \right\| _{H^{-\varepsilon}}}{(t-s)^{\lambda}} \right] 
\xrightarrow{h \searrow
0}0.
\end{align*}
On the other hand, for any $R>0$,
$$
\sup _{N\in {\mathbb N}} {\mathbb P}  \left( \left\| \bar{\bf\Phi}^N_0 \right\| _{H^{-\varepsilon}} >R \right) 
\leq \frac{1}{R} \sup _{N\in {\mathbb N}} \mathbb{E}\left[ \left\| \xi_N \right\| _{H^{-\varepsilon}} \right]
\xrightarrow{R\to\infty}0.
$$
Since $H^{-\varepsilon}$ is compactly embedded in $H^{-\varepsilon'}$ for any $\varepsilon'>\varepsilon$, we can conclude that $\{\bar{\bf\Phi}^N\}$ is tight in $C([0,T];H^{-\varepsilon'})$.

Let $f$ be a bounded continuous function on $H^{-\varepsilon}$. Then, by the invariance of 
$\mu_N^{(\alpha)}$ with respect to $\bar{\bf\Phi}^N$, we have
\begin{align*}
\mathbb{E}[f(\bar{\bf\Phi}_t)]
=\lim_{k\to\infty}\mathbb{E}[f(\bar{\bf\Phi}^{N_k})]
=\lim_{k\to\infty}\int _{{\mathcal D}'(\Lambda )} f(\phi) \mu_{N_k}^{(\alpha)}(d\phi)
=\int  _{{\mathcal D}'(\Lambda )} f(\phi) \mu^{(\alpha)}(d\phi)
\end{align*}
for any $t\ge0$.
\end{proof}

\subsection{Uniqueness of the limit}

Let $\Phi=\Phi(\phi)$ be the strong solution of \eqref{expsqe} with the initial value $\phi$, as in Theorem \ref{mainthm1}.
Let $\xi$ be a $\mathcal{D}'(\Lambda)$-valued random variable which has the law $\mu^{(\alpha )}$ and is independent of $W$, and define
$$
\bar{\Phi}:=\Phi(\xi).
$$
We show that the law of $\bar{\Phi}$ is the unique limit of $\{\bar{\bf\Phi}^N\}$.

\begin{proof}[{\bfseries Proof of Theorem \ref{mainthm2}}]
We show that
$$
\bar{\bf\Phi}^N\xrightarrow{d}\bar{\Phi}
$$
in $C([0,T];H^{-\varepsilon})$.
Since the family $\{\xi_N\}_{N\in\mathbb{N}}\cup\{\xi\}$ is independent of $W$, we regard the probability space $(\Omega,\mathbb{P})$ as a product space $(\Omega_1\times\Omega_2,\mathbb{P}_1\otimes\mathbb{P}_2)$ where $\{\xi_N\}_{N\in\mathbb{N}}\cup\{\xi\}$ are defined on $(\Omega_1,\mathbb{P}_1)$ and $W$ is defined on $(\Omega_2,\mathbb{P}_2)$.
Since $\xi_N\xrightarrow{d}\xi$ in $H^{-\varepsilon}$ (Corollary \ref{cor:expmeas}), by the Skorokhod's representation theorem, there is another probability space $(\hat{\Omega}_1,\hat{\mathbb{P}}_1)$ and random variables $\hat{\xi}_N$ and $\hat{\xi}$, such that $\hat{\xi}_N\overset{d}{=}\xi_N$, $\hat{\xi}\overset{d}{=}\xi$, and
$$
\hat{\xi}_N\to\hat{\xi}
$$
in $H^{-\varepsilon}$ almost surely.
Let $\hat{\bf\Phi}^N={\bf\Phi}(\hat{\xi}_N)$ be the solution of \eqref{expsqe2} with an initial value $\hat{\xi}_N$, and let $\hat{\Phi}=\Phi(\hat{\xi})$. The stochastic processes $\hat{\bf\Phi}^N$ and $\hat{\Phi}$ are defined on the space $(\hat{\Omega}_1\times\Omega_2,\hat{\mathbb{P}}_1\otimes\mathbb{P}_2)$.
Since $\hat{\bf\Phi}^N\overset{d}{=}\bar{\bf\Phi}^N$ and $\hat{\Phi}\overset{d}{=}\bar{\Phi}$, it is sufficient to show that
$$
\hat{\bf\Phi}^N\to\hat{\Phi}
$$
in $C([0,T];H^{-\varepsilon})$, in probability.

We decompose $\hat{\bf\Phi}^N=X(\hat{\xi}_N)+\hat{\bf Y}^N$ similarly to \eqref{eq:statY}, and decompose $\hat{\Phi}=\hat{X}+\hat{Y}$, where
$$
\hat{X}=X(\hat{\xi}),\quad
\hat{Y}=\mathcal{S}(0,\mathcal{X}^{(\exp,\infty)}(\hat{\xi})).
$$
For the OU terms, we have
\begin{align*}
\|X(\hat{\xi}_N)-X(\hat{\xi})\|_{C([0,T];H^{-\varepsilon})}
=\sup_{t\in[0,T]}\|e^{\frac12(\triangle-1)t}(\hat{\xi}_N-\hat{\xi})\|_{H^{-\varepsilon}}
\le\|\hat{\xi}_N-\hat{\xi}\|_{H^{-\varepsilon}}
\xrightarrow{N\to\infty}0,
\end{align*}
almost surely.
For the remainders, we consider the deterministic initial value problem
\begin{align*}
\left\{
\begin{aligned}
\partial_t\Upsilon_t^N&=\frac12(\triangle-1)\Upsilon_t^N-\frac\alpha2 P_N\left(e^{\alpha P_N\Upsilon_t^N}\mathcal{X}_t^N\right),\\
\Upsilon_0^N&=\upsilon^N,
\end{aligned}
\right.
\end{align*}
for $\mathcal{X}^N\in C([0,T];C_+(\Lambda))$ and $\upsilon^N\in H^{2-\beta}$.
Denote the unique classical global solution by $\Upsilon^N=\mathcal{S}_N(\upsilon^N,\mathcal{X}^N)$.
Similarly to the proof of Lemma \ref{lem:existence}, we can show that, if
$$
\upsilon^N\to\upsilon\quad\text{in $H^{2-\beta}$},\quad
\mathcal{X}^N\to\mathcal{X}\quad\text{in $L^2([0,T];H_+^{-\beta})$},
$$
then one has
$$
\mathcal{S}_N(\upsilon^N,\mathcal{X}^N)
\to\mathcal{S}(\upsilon,\mathcal{X})
\quad\text{in $L^2([0,T];H^{1+\delta})\cap C([0,T];H^\delta)$}
$$
for any $\delta\in(0,1-\beta)$. By using this fact, to show the convergence $\hat{\bf Y}^N\to\hat{Y}$ in probability, it is sufficient to show that
$$
\mathcal{X}^{(\exp,N)}(\hat{\xi}_N)\to\mathcal{X}^{(\exp,\infty)}(\hat{\xi})
$$
in probability. This is a consequence of Lemma \ref{lem:contiexpOU}.
\end{proof}
%
%
%
%
\section{Proof of Theorem \ref{mainthm3}} \label{sec:DF}
In this section, we give a proof of Theorem \ref{mainthm3}. We fix
$\beta\in (\frac{\alpha^{2}}{4\pi},1)$ and set 
$D={\rm Span}\{e_{k};
k\in {\mathbb Z}^{2} \}$,
$H=L^{2}(\Lambda)$ and $E=H^{-\beta}(\Lambda)$.
In what follows, $\langle \cdot, \cdot \rangle$ stands for the pairing of $E$ and its dual space
$E^{*}=H^{\beta}(\Lambda)$.
By Theorem \ref{thm:expgff}, there exists a ${\mathcal B}(E)/{\mathcal B}(E)$-measurable map
which extends ${\rm exp}^{\diamond}(\alpha \cdot) \in L^{2}(\mu_{0}; E)$. We also denote it
by ${\rm exp}^{\diamond}(\alpha \cdot)$.
Let 
$({\mathcal E}, {\mathfrak F}C^{\infty}_{b})$ be the 
pre-Dirichlet form defined by (\ref{DF-intro}).
Applying the integration by parts formula for the ${\rm exp}(\Phi)_{2}$-measure $\mu^{(\alpha)}$
as in \cite{AR91, AKMR19}, we have
\begin{equation*}
{\cal E}(F,G)=-\int_{E} {\cal L}F(\phi) G(\phi) \mu^{(\alpha)}(d\phi),
\quad F,G\in {\mathfrak F}C_{b}^{\infty},
\label{IbP}
\end{equation*}
where
\begin{align*}
\mathcal{L}F (\phi )
&=\frac12\sum_{i, j=1}^n \partial_i \partial_jf(\langle \phi,{l_1}\rangle,\dots,\langle \phi,{l_n}\rangle) \langle l_{i}, l_{j} \rangle
\\
&\quad-\frac12\sum_{j=1}^n \partial_jf(\langle \phi,{l_1}\rangle,\dots,\langle \phi,{l_n}\rangle)
\cdot \big \{
{\big \langle} (1-\triangle)\phi, l_{j}
{\big \rangle}
+\alpha
{\big \langle}
\exp^\diamond(\alpha \phi), {l_j} {\big \rangle}
\big \}
\end{align*}
for $F(\phi)=f(\langle \phi,{l_1}\rangle,\dots,\langle \phi,{l_n}\rangle)$
with $f\in C^{\infty}_{b}(\mathbb R^{n}), l_{1}, \ldots l_{n} \in 
D$.
Note that Theorem \ref{thm:expgff} and Corollary \ref{cor:expmeas} imply ${\mathcal L}F\in L^{2}(\mu ^{(\alpha)})$.
This formula implies that $({\mathcal E}, {\mathfrak F}C^{\infty}_{b})$ is closable on $L^{2}(\mu^{(\alpha)})$.
We denote the closure of (\ref{DF-intro})
by
$({\mathcal E}, {\mathcal D}({\mathcal E}))$.
As mentioned in Section 1.2, $({\mathcal E}, {\mathcal D}({\mathcal E}))$
is a quasi-regular Dirichlet form 
on $L^{2}(\mu^{(\alpha)})$
and we obtain an $E$-valued diffusion
process 
${\mathbb M}=(\Theta, {\mathcal G}, ( {\mathcal G}_{t})_{t\geq 0},  (\Psi_{t})_{t\geq 0}, ({\mathbb Q}_{\phi} )_{\phi \in E})$
properly associated with $({\cal E}, {\cal D(E)})$. By recalling Theorem \ref{thm:expgff} and applying
\cite[Lemma 4.2]{AR91},
we have 
\begin{equation}
{\mathbb E}^{{\mathbb Q}_{\phi}} \Big[
\int_{0}^{T} \Vert \exp^{\diamond}(\alpha \Psi_{t}) \Vert^{2}_{E} dt \Big ]<\infty, \quad T>0, ~\mu^{(\alpha)}
\mbox{-a.e.}~\phi.
\label{Q-1}
\end{equation}
Then (\ref{Q-1}) implies
\begin{equation}
{\mathbb Q}_{\phi} \Big( \int_{0}^{T} 
\Vert \exp^{\diamond}(\alpha \Psi_{t})\Vert_{E} dt<\infty~\mbox{for all }T>0 \Big)=1, \quad \mu^{(\alpha)}
\mbox{-a.e.}~\phi.
\label{Q-2}
\end{equation}
Thus we may apply \cite[Lemma 6.1 and Theorem 6.2]{AR91}, which implies that there exists
a family of independent one-dimensional $({\mathcal G}_{t})$-Brownian motions $\{b^{(k)}
=(b^{(k)}_{t})_{t\geq 0} \}_{k\in \mathbb Z^{2}}$ defined on 
$(\Theta, {\mathcal G}, {\mathbb Q}_{\phi})$ such that
\begin{equation}
\label{Q-weak form-1}
\begin{array}{rl}
\langle \Psi_{t}, e_{k} \rangle=&{\displaystyle{
\langle \phi, e_{k} \rangle+b_{t}^{(k)}
+\frac{1}{2} \int_{0}^{t} \big \langle \Psi_{s}, (\Delta-1)e_{k} \big \rangle ds
}}
\\
&{\displaystyle{-
\frac{\alpha}{2}
 \int_{0}^{t} \big \langle  \exp^{\diamond}(\alpha \Psi_{s}), e_{k} \big \rangle ds, \quad t\geq 0,~{\mathbb Q}_{\phi}
 \mbox{-a.s.},~\mu^{(\alpha)}\mbox{-a.e. } \phi
}}
\end{array}
\end{equation}
for each $k\in {\mathbb Z}^{2}$.
Hence there exists 
an $H$-cylindrical 
$( {\mathcal G}_{t})$-Brownian motion ${\mathcal W}=({\mathcal W}_{t})_{t\geq 0}$
defined on $(\Theta, {\mathcal G}, {\mathbb Q}_{\phi})$ such that
\begin{equation}
\label{Q-weak form}
\begin{array}{rl}
\langle \Psi_{t}, l \rangle=&{\displaystyle{
\langle \phi, l \rangle+\langle {\mathcal W}_{t}, l \rangle
+\frac{1}{2} \int_{0}^{t} \big \langle \Psi_{s}, (\Delta-1)l \big \rangle ds
}}
\\
&{\displaystyle{-
\frac{\alpha}{2}
 \int_{0}^{t} \big \langle  \exp^{\diamond}(\alpha \Psi_{s}), l \big \rangle ds, \quad t\geq 0,~
 l\in D,~
 {\mathbb Q}_{\phi}
 \mbox{-a.s.},~\mu^{(\alpha)}\mbox{-a.e. } \phi.
 }}
\end{array}
\end{equation}
By noting that
$D$ is dense in ${\rm Dom}(\Delta)=H^{2}(\Lambda)$
and 
(\ref{Q-2}), we may apply \cite[Theorem 13]{Ond04}, and thus
we have that
(\ref{Q-weak form}) is equivalent to the mild form of the SPDE (\ref{expsqe}), i.e.,
\begin{equation}
{\Psi}_t=e^{\frac12(\triangle-1)t}\phi
-\frac\alpha2\int_0^te^{\frac12(\triangle-1)(t-s)}\exp^\diamond(\alpha{\Psi}_s)ds
+\int_0^te^{\frac12(\triangle-1)(t-s)}d{\mathcal W}_s, \quad t\geq 0.
\label{Q-mild}
\end{equation}


Now we are going to prove 
that the weak solution $({\Psi},{\mathcal W})$ coincides with the strong solution $\Phi$ driven by ${\mathcal W}$.
We need prepare the following two lemmas.

\begin{lem}[{\cite[Corollary~2.91]{BCD11}}]\label{lem:f(u)}
Let $f$ be a smooth function on $\mathbb{R}$. Let $s>0$ and $p,q\in[1,\infty]$.
For any $u\in B_{p,q}^s\cap L^\infty$, the function $f(u)$ belongs to $B_{p,q}^s\cap L^\infty$.
Moreover, the mapping
$$
B_{p,q}^s\cap L^\infty\ni u\mapsto
f(u)\in B_{p,q}^s\cap L^\infty
$$
is Lipschitz on any bounded set $\{u;\|u\|_{B_{p,q}^s\cap L^\infty}\le K\}$ for $K>0$.
\end{lem}

\begin{lem}\label{lem:wexplaw}
Let $\phi\in E$ and $f\in H^{1+\delta}$.
Assume that the following convergences hold.
\begin{align*}
\exp^\diamond(\alpha\phi)&=\lim_{N\to\infty}\exp_N^\diamond(\alpha\phi)
\quad\text{in $H^{-\beta}$},\\
\exp^\diamond(\alpha(f+\phi))&=\lim_{N\to\infty}\exp_N^\diamond(\alpha(f+\phi))
\quad\text{in $B_{1,1}^{-\beta}$}.
\end{align*}
Then one has the equality
$$
\exp^\diamond(\alpha(f+\phi))
=e^{\alpha f}\exp^\diamond(\alpha\phi).
$$
\end{lem}

\begin{proof}
Since $\exp_N^\diamond(\alpha(f+\phi))=\exp(\alpha P_Nf)\exp_N^\diamond(\alpha\phi)$ by definition, we have
\begin{align*}
&\|\exp_N^\diamond(\alpha(f+\phi))-e^{\alpha f}\exp^\diamond(\alpha\phi)\|_{B_{1,1}^{-\beta}}\\
&\lesssim\|\exp(\alpha P_Nf)-\exp(\alpha f)\|_{H^\beta}\|\exp_N^\diamond(\alpha\phi)\|_{H^{-\beta}}\\
&\quad+\|\exp(\alpha f)\|_{C(\Lambda)}\|\exp_N^\diamond(\alpha\phi)-\exp^\diamond(\alpha\phi)\|_{H^{-\beta}}
\end{align*}
by Lemma \ref{thm:func*dist} and Theorem \ref{thm:stableM}.
The second term in the right hand side converges to $0$ by assumption.
For the first term, since $P_Nf$ is uniformly bounded in $H^{1+\delta}$,
\begin{align*}
\|\exp(\alpha P_Nf)-\exp(\alpha f)\|_{H^\beta}
\lesssim_f\|P_Nf-f\|_{H^\beta\cap C(\Lambda)}
\lesssim\|P_Nf-f\|_{H^{1+\delta}}
\end{align*}
by Lemma \ref{lem:f(u)}.
Since $f\in H^{1+\delta}$, we have that $\lim_{N\to\infty}\|P_Nf-f\|_{H^{1+\delta}}=0$.
Therefore we have the required equality.
\end{proof}

Let ${\mathfrak X}={\mathfrak X}(\phi)$ be the OU process 
driven by $\mathcal W$ with an initial value ${\mathfrak X}_0=\phi\in E$. 

\begin{thm}
For any $\mu^{(\alpha)}$-a.e. $\phi\in E$, the equality
$$
{\Psi}={\mathfrak X}(\phi)+\mathcal{S}(0,\exp^\diamond(\alpha{\mathfrak X}(\phi))),
$$
holds ${\mathbb Q}_\phi$-almost surely, and hence Theorem \ref{mainthm3} follows.
\end{thm}
\begin{proof}
We decompose ${\Psi}={\mathfrak X}(\phi)+{\mathfrak Y}$. 
For $\mu^{(\alpha)}$-a.e. $\phi$, 
${\Psi}$ solves the mild equation (\ref{Q-mild}).
The second term on the right-hand side of 
(\ref{Q-mild}) is nothing but the remainder ${\mathfrak Y}$.
To show the result, it is sufficient to show that
$$
{\mathbb Q}_\phi\Bigl({\mathfrak Y}=\mathcal{S}(0,\exp^\diamond(\alpha{\mathfrak X}(\phi)))\Bigr)=1,
\quad\text{$\mu^{(\alpha)}$-a.e. $\phi$.}
$$
By the invariance of $\mu^{(\alpha)}$ under ${\Psi}$ and Lemma \ref{lem:contiexpOU},
\begin{align*}
\int_E
{\mathbb E}^{{\mathbb Q}_{\phi}}
{\Big [}
\|\exp^\diamond(\alpha{\Psi})\|_{L^2([0,T];H^{-\beta})}^2 
{\Big ]}
\mu^{(\alpha)}(d\phi)
=\int_0^Tdt\int_E\|\exp^\diamond(\alpha\phi)\|_{H^{-\beta}}^2\mu^{(\alpha)}(d\phi)
<\infty.
\end{align*}
In particular,
$$
{\mathbb Q}_\phi\Bigl (
\exp^\diamond(\alpha{\Psi})\in L^2([0,T];H^{-\beta})\Bigr )=1,
\quad\text{$\mu^{(\alpha)}$-a.e. $\phi$.}
$$
Then since ${\mathfrak Y}$ belongs to $L^2([0,T];H^{1+\delta})\cap C([0,T];H^\delta)$ by 
the Schauder estimate, we can check that
$$
{\mathbb Q}_\phi
\big (
{\mathfrak Y}\in\mathscr{Y}_T\big )=1,
\quad\text{$\mu^{(\alpha)}$-a.e. $\phi$,}
$$
similarly to the proof of Lemma \ref{lem:existence}.
Let $A$ be  the set of all $\phi\in E$ such that the convergence
$$
\exp^\diamond(\alpha\phi)
=\lim_{N\to\infty}\exp_N^\diamond(\alpha\phi)
$$
holds in $H^{-\beta}$. Recall that $\mu_0(A)=1$, so $\mu^{(\alpha)}(A)=1$ 
by the absolute continuity (see Corollary \ref{cor:expmeas}).
By using the invariance of $\mu^{(\alpha)}$ under ${\Psi}$,
\begin{align*}
\int_E{\mathbb E}^{{\mathbb Q}_{\phi}}
\left[\int_0^T\mathbf{1}_{A^c}({\Psi}_t)dt\right]\mu^{(\alpha)}(d\phi)
=\int_0^T\int_E\mathbf{1}_{A^c}(\phi)\mu^{(\alpha)}(d\phi)
=T\mu^{(\alpha)}(A^c)=0.
\end{align*}
Similarly, by the invariance of $\mu_0$ under ${\mathfrak X}$,
\begin{align*}
\int_E
{\mathbb E}^{{\mathbb Q}_\phi}
\left[\int_0^T\mathbf{1}_{A^c}({\mathfrak X}_t)dt\right]\mu^{(\alpha)}(d\phi)
&\lesssim \int_E
{\mathbb E}^{{\mathbb Q}_\phi}
\left[\int_0^T\mathbf{1}_{A^c}({\mathfrak X}_t)dt
\right]\mu_0(d\phi)\\
&=\int_0^T\int_E\mathbf{1}_{A^c}(\phi)\mu_0(d\phi)
=T\mu_0(A^c)=0.
\end{align*}
As a result,
\begin{align*}
{\mathbb Q}_\phi \Big(
{\Psi}_t\in A,\ {\mathfrak X}_t\in A,\ \text{a.e. $t$}
\Big )
=1,
\quad\text{$\mu^{(\alpha)}$-a.e. $\phi$.}
\end{align*}
Since ${\mathfrak Y}\in L^2([0,T];H^{1+\delta})$ 
holds ${\mathbb Q}_\phi$-almost surely, we additionally get
\begin{align*}
{\mathbb Q}_\phi
\Big(
{\Psi}_t\in A,\ {\mathfrak X}_t\in A,\ {\mathfrak Y}_t\in H^{1+\delta},\ \text{a.e. $t$}
\Big )
=1.
\end{align*}
Hence by noting that ${\Psi} = {\mathfrak X}(\phi)+{\mathfrak Y}$ and 
applying Lemma \ref{lem:wexplaw}, we have
$$
{\mathbb Q}_\phi
\Big (
\exp^\diamond(\alpha{\Psi}_t)=e^{\alpha{\mathfrak Y}_t}\cdot\exp^\diamond(\alpha{\mathfrak X}_t),
\ \text{a.e. $t$}
\Big )=1,
\quad\text{$\mu^{(\alpha)}$-a.e. $\phi$},
$$
which yields that ${\mathfrak Y}$ is a mild solution of 
\eqref{eq:Ups} with $(\upsilon,\mathcal{X})=(0,\exp^\diamond(\alpha{\mathfrak X}))$.
\end{proof}
%
%
%
\vspace{4mm}
\noindent
{\bf Acknowledgements.}
The authors thank anonymous referees for helpful suggestions that 
improved the quality of the present paper.
This work was partially supported by JSPS KAKENHI Grant Numbers 17K05300, 17K14204 
and 19K14556.

\renewcommand{\thesection}{\Alph{section}}

\setcounter{section}{0}

\section{Appendix}

\subsection{Besov space}

Let $(\chi,\rho)$ be a dyadic partition of unity, i.e., they are smooth radial functions on $\mathbb{R}^2$ such that,
\begin{itemize}
\item $0\le\chi\le1$, $0\le\rho\le1$,
\item $\chi$ is supported in $\{x;|x|\le\frac43\}$, $\rho$ is supported in $\{x;\frac34\le|x|\le\frac83\}$,
\item $\chi(\xi)+\sum_{j=0}^\infty\rho(2^{-j}\xi)=1$ for any $\xi\in\mathbb{R}^2$.
\end{itemize}
Denote $\rho_{-1}=\chi$ and $\rho_j=\rho(2^{-j}\cdot)$ for $j\ge0$. Define
$$
\Delta_jf=\sum_{k\in\mathbb{Z}^2}\rho_j(k)\langle f, {\bf e}_k\rangle {\bf e}_k.
$$
For $s\in\mathbb{R}$ and $p,q\in[1,\infty]$, we define the inhomogeneous Besov norm
$$
\|f\|_{B_{p,q}^s}:=\left\|\{2^{js}\|\Delta_jf\|_{L^p(\Lambda)}\}_{j\ge-1}\right\|_{\ell^q}.
$$

\begin{prop}[{\cite[Page 99]{BCD11}}]
For any $s\in\mathbb{R}$, $H^s=B_{2,2}^s$.
\end{prop}

\subsection{Schauder estimates}

\begin{prop}[{\cite[Propositions~5 and 6]{MW17a}}]\label{prop:heatsemigr}
Let $s\in\mathbb{R}$, $p,q\in[1,\infty]$ and $\mu>0$.
\begin{enumerate}
\item \label{prop:heatsemigr1} For every $\delta\ge0$, $\|e^{\frac12(\triangle-1)t}u\|_{B_{p,q}^{s+2\delta}}\lesssim t^{-\delta}\|u\|_{B_{p,q}^s}$ uniformly over $t>0$.
\item \label{prop:heatsemigr2} For every $\delta\in[0,1]$, $\|(e^{\frac12(\triangle-1)t}-1)u\|_{B_{p,q}^{s-2\delta}}\lesssim t^\delta\|u\|_{B_{p,q}^s}$ uniformly over $t>0$.
\end{enumerate}
\end{prop}

\begin{prop}\label{prop:Schauder}
Let $u$ solve the equation (in the mild form)
\begin{align*}
\left\{
\begin{aligned}
\partial_tu(t,x)&=\frac12(\triangle-1)u(t,x)+U(t,x),\quad t>0,\quad x\in\Lambda,\\
u(0,\cdot)&=u_0,\quad x\in\Lambda.
\end{aligned}
\right.
\end{align*}
Let $r\in(1,\infty]$ and define $r'\in[1,\infty)$ by $1/r+1/r'=1$.
Then for any $p,q\in[1,\infty]$, $\theta\in\mathbb{R}$, $\varepsilon>0$, and $\eta\in(0,2/r')$, one has
$$
\|u\|_{L^r([0,T];B_{p,q}^{\theta+2-\varepsilon})\cap C([0,T];B_{p,q}^{\theta+2/r'-\varepsilon})\cap C^{\eta/2}([0,T];B_{p,q}^{\theta+2/r'-\varepsilon-\eta})}
\lesssim\|u_0\|_{B_{p,q}^{\theta+2-\varepsilon}}+\|U\|_{L^r([0,T];B_{p,q}^{\theta})}.
$$
In particular, for $\beta\in(0,1)$ and $\delta\in(0,1-\beta)$, setting $r=p=q=2$, $\theta = -\beta$, $\varepsilon = 1-\beta -\delta$ and $\eta = \delta$, one has
$$
\|u\|_{L^2([0,T];H^{1+\delta})\cap C([0,T];H^{\delta})\cap C^{\delta/2}([0,T];L^2)}
\lesssim\|u_0\|_{H^{1+\delta}}+\|U\|_{L^2([0,T];H^{-\beta})}.
$$
\end{prop}

\begin{proof}
We decompose
\begin{align*}
&u_t=e^{\frac12t(\triangle-1)}u_0+\int_0^te^{\frac12(t-s)(\triangle-1)}U_sds
=:u_t^0+u_t^1,\\
&u_t-u_s=(e^{\frac12(t-s)(\triangle-1)}-1)u_s+\int_s^te^{\frac12(t-v)(\triangle-1)}U_vdv
=:u_{ts}^0+u_{ts}^1.
\end{align*}
(1) {\bf Bound in $L^r([0,T];B_{p,q}^{\theta+2-\varepsilon})$.}
By Proposition \ref{prop:heatsemigr}-(i),
\begin{align*}
\|u_t^0\|_{B_{p,q}^{\theta+2-\varepsilon}}
\lesssim\|u_0\|_{B_{p,q}^{\theta+2-\varepsilon}},\quad
\|u_t^1\|_{B_{p,q}^{\theta+2-\varepsilon}}
\lesssim\int_0^t(t-s)^{-\frac{2-\varepsilon}2}\|U_s\|_{B_{p,q}^{\theta}}ds.
\end{align*}
By Young's inequality,
\begin{align*}
\|u^1\|_{L^r([0,T];B_{p,q}^{\theta+2-\varepsilon})}
\lesssim\|t\mapsto t^{-\frac{2-\varepsilon}2}\|_{L^1([0,T])}
\|U\|_{L^r([0,T];B_{p,q}^{\theta})}
\lesssim\|U\|_{L^r([0,T];B_{p,q}^{\theta})}.
\end{align*}
(2) {\bf Bound in $L^\infty([0,T];B_{p,q}^{\theta+2/r'-\varepsilon})$.}
By Proposition \ref{prop:heatsemigr}-(i),
\begin{align*}
\|u_t^1\|_{B_{p,q}^{\theta+2/r'-\varepsilon}}
\lesssim\int_0^t(t-s)^{-(\frac1{r'}-\frac\varepsilon2)}\|U_s\|_{B_{p,q}^{\theta}}ds.
\end{align*}
By Young's inequality,
\begin{align*}
\|u^1\|_{L^\infty([0,T];H^{\theta+2-2/r-\varepsilon})}
\lesssim\|t\mapsto t^{-(\frac1{r'}-\frac\varepsilon2)}\|_{L^{r'}([0,T])}
\|U\|_{L^r([0,T];B_{p,q}^{\theta})}
\lesssim\|U\|_{L^r([0,T];B_{p,q}^{\theta})}.
\end{align*}
(3) {\bf Bound in $C^{\eta/2}([0,T];B_{p,q}^{\theta+2/r'-\varepsilon-\eta})$.}
By Proposition \ref{prop:heatsemigr}-(ii) and the bound in $L^\infty([0,T];B_{p,q}^{\theta+2/r'-\varepsilon})$,
\begin{align*}
\|u_{ts}^0\|_{B_{p,q}^{\theta+2/r'-\varepsilon-\eta}}
\lesssim(t-s)^{\eta/2}\|u_s\|_{B_{p,q}^{\theta+2/r'-\varepsilon}}
\lesssim(t-s)^{\eta/2}\|U\|_{L^r([0,T];B_{p,q}^{\theta})},
\end{align*}
and by Proposition \ref{prop:heatsemigr}-(i),
\begin{align*}
\|u_{ts}^1\|_{B_{p,q}^{\theta+2/r'-\varepsilon-\eta}}
&\lesssim\int_s^t(t-v)^{-(\frac1{r'}-\frac\varepsilon2-\frac\eta2)}\|U_v\|_{B_{p,q}^{\theta}}dv\\
&\lesssim\left(\int_s^t(t-v)^{-1+(\varepsilon+\eta)\frac{r'}{2}}dv\right)^{\frac1{r'}}\left(\int_s^t\|U_v\|_{B_{p,q}^{\theta}}^rdv\right)^{\frac1r}\\
&\lesssim(t-s)^{(\varepsilon+\eta)/2}\|U\|_{L^r([0,T];B_{p,q}^{\theta})}.
\end{align*}
Then $u\in C^{\eta/2}([0,T];B_{p,q}^{\theta+2/r'-\varepsilon-\eta})$ implies $u\in C([0,T];B_{p,q}^{\theta+2/r'-\varepsilon-\eta})$. Since $\eta,\varepsilon>0$ are arbitrary small, one has $u\in C([0,T];B_{p,q}^{\theta+2/r'-\varepsilon})$ for any $\varepsilon>0$.
\end{proof}


\end{document}